\documentclass[10pt]{article}
\usepackage{lmodern}
\usepackage{amsmath}
\usepackage[T1]{fontenc}
\usepackage[utf8]{inputenc}
\usepackage{authblk}
\usepackage{amsfonts}
\usepackage{graphicx}
\usepackage{rotating}
\usepackage{amssymb}
\usepackage[english]{babel}
\usepackage{color}
\usepackage{amsthm}
\usepackage{graphicx}
\usepackage{mathrsfs}
\usepackage{makecell}
\usepackage{microtype}
\usepackage{enumerate}
\usepackage{mathscinet}
\usepackage{array}
\usepackage{multirow}
\usepackage{booktabs}
\usepackage{enumerate}
\usepackage{booktabs}
\usepackage[cal=boondoxo,bb=ams]{mathalfa}
\usepackage{hyperref}
\hypersetup{hidelinks}
\usepackage{titlesec}
\newtheorem{theorem}{Theorem}[section]
\newtheorem{prop}{Proposition}[section]

\newtheorem{coro}{Corollary}[section]
\newtheorem{remark}{Remark}[section]

\newcommand{\ml}{\mathcal}
\newcommand{\mb}{\mathbb}

\DeclareMathOperator{\non}{non}
\DeclareMathOperator{\lin}{lin}
\DeclareMathOperator{\intt}{int}
\DeclareMathOperator{\extt}{ext}
\DeclareMathOperator{\bdd}{bdd}

\title{Asymptotic behaviors for the Jordan-Moore-Gibson-Thompson equation in the viscous case}
\author[1,2]{Wenhui Chen\thanks{Wenhui Chen (wenhui.chen.math@gmail.com)}}
\affil[1]{School of Mathematics and Information Science, Guangzhou University, 510006 Guangzhou, China}
\affil[2]{School of Mathematical Sciences, Shanghai Jiao Tong University, 200240 Shanghai, China}
\author[3]{Hiroshi Takeda\thanks{Hiroshi Takeda (h-takeda@fit.ac.jp)}}
\affil[3]{Department of Intelligent Mechanical Engineering, Faculty of Engineering, Fukuoka Institute of Technology,  811-0295 Fukuoka, Japan}

\date{}
\begin{document}

\maketitle
\begin{abstract}
In this paper, we study large-time behaviors for a fundamental model in nonlinear acoustics, precisely, the viscous Jordan-Moore-Gibson-Thompson (JMGT) equation in the whole space $\mathbb{R}^n$. This model describes nonlinear acoustics in perfect gases under irrotational flow and equipping Cattaneo's law of heat conduction. By employing refined WKB analysis and Fourier analysis, we derive first- and second-order asymptotic profiles of solution to the Moore-Gibson-Thompson (MGT) equation as $t\gg 1$, which illustrates novel optimal estimates for the solutions even subtracting its profiles. Concerning the nonlinear JMGT equation, via suggesting a new decomposition of nonlinear portion, we investigate the existence and large-time profiles of global (in time) small data Sobolev solutions with suitable regularity. These results help bridge a new connection between the JMGT equation and diffusion-waves as $t\gg1$.\\
	
	
	\noindent\textbf{AMS Classification (2020)} Primary: 35L75; Secondary: 35L30, 35A01, 35B40.
\end{abstract}
\fontsize{12}{15}
\selectfont

\section{Introduction}
\subsection{Background for the JMGT equation}
It is widely known that nonlinear acoustics arise in numerous applications, for example, high-intensity ultrasonic waves have been applied in medical imaging and therapy, ultrasound cleaning and welding (see \cite{Abramov-1999,Dreyer-Krauss-Bauer-Ried-2000,Kaltenbacher-Landes-Hoffelner-Simkovics-2002} and references therein). This necessitates a deep comprehend of nonlinear acoustic models as well as their analytical behaviors. There are several mathematical models for describing nonlinear acoustics phenomena, e.g. Kuznetsov's equation, Blackstock's model, the Jordan-Moore-Gibson-Thompson equation, and the Khokhlov-Zabolotskaya-Kuznetsov equation, etc. In the present paper, we will study some large-time behaviors for the Jordan-Moore-Gibson-Thompson equation and its linearized model in the whole space $\mb{R}^n$.

To characterize propagations of sound in thermoviscous fluids, the classical theory of nonlinear acoustics considers the mathematical model from the irrotational Navier-Stokes-Fourier equations with the Lighthill scheme of approximation procedures.  To be specific, let us amalgamate the conservation of mass, momentum and energy associated with Fourier's law of heat conduction
\begin{align}\label{Fourier-law}
	\mathbf{q}=-\kappa\nabla \Theta,
\end{align} 
with the thermal conductivity $\kappa>0$, where $\mathbf{q}$ and $\Theta$ denote, respectively, the heat flux vector and the absolute temperature. Moreover, by considering the state equation for perfect gases and neglecting higher-order terms (i.e. the Lighthill scheme of approximation procedures) in the derivations when combining the governing equations under the irrotational flow, we can derive one of the most established models in nonlinear acoustics, i.e. Kuznetsov's equation as follows:
\begin{align}\label{Kuz-Eq}
	\varphi_{tt}-c_0^2\Delta\varphi-\delta\Delta\varphi_t=\partial_t\left(\frac{B}{2A}(\varphi_t)^2+|\nabla\varphi|^2\right),
\end{align}
where the scalar unknown $\varphi=\varphi(t,x)\in\mb{R}$ denotes the acoustic velocity potential under Fourier's law of heat conduction. The constant coefficient $c_0>0$ of the Laplacian is the speed of sound in the undisturbed gas, $\delta>0$ is referred to the diffusivity of sound, and the constant ratio $B/A>0$ indicates the nonlinearity of the equation of state for a given fluid.  

Nevertheless, the application of Fourier's law \eqref{Fourier-law} in acoustic waves leads to an infinite signal speed paradox. It seems to be incompatible in phenomena of wave propagations. Therewith, Cattaneo's law of heat conduction 
\begin{align}\label{Cattaneo-law}
	\tau\mathbf{q}_t+\mathbf{q}=-\kappa\nabla \Theta,
\end{align}
where the positive constant $\tau$ designates the small thermal relaxation, i.e. $0<\tau\ll 1$, can be used instead to eliminate this paradox. The combination of three conservation laws carrying Cattaneo's law \eqref{Cattaneo-law}  and the state equation leads to the third-order (in time) nonlinear hyperbolic equation
\begin{align}\label{JMGT-ALL}
	\tau \psi_{ttt}+\psi_{tt}-c_0^2\Delta \psi-(\delta+\tau c_0^2)\Delta \psi_t=\partial_t\left(\frac{B}{2A}(\psi_t)^2+|\nabla\psi|^2\right)
\end{align}
with the acoustic velocity potential $\psi=\psi(t,x)\in\mb{R}$ under Cattaneo's law of heat conduction. Namely, as mentioned in \cite{Dell-Lasiecka-Pata}, the equation \eqref{JMGT-ALL} can be derived by the Navier-Stokes-Cattaneo equations with irrotational assumption. In the literature, this model \eqref{JMGT-ALL} is the so-called Jordan-Moore-Gibson-Thompson (JMGT) equation and we refer to \cite{Jordan-2014} for its derivation in detail. 

\subsection{Mathematical researches on the JMGT equation}
 Before recalling the recent studies for the nonlinear JMGT equation, let us begin with considering the corresponding linearized model to \eqref{JMGT-ALL}, i.e. the Moore-Gibson-Thompson (MGT) equation
\begin{align}\label{MGT-ALL}
	\tau \psi_{ttt}+\psi_{tt}-c_0^2\Delta \psi-(\delta+\tau c_0^2)\Delta \psi_t=0,
\end{align}
which has caught a lot of attentions in recent years (see \cite{Kaltenbacher-Lasiecka-Marchand-2011,Marchand-McDevitt-Triggiani-2012,Conejero-Lizama-Rodenas-2015,Dell-Pata-2017,Pellicer-Said-2017,B-L-2019,Chen-Ikehata=2021} and references therein). We underline that $\delta=0$ is the threshold to distinguish some crucial behaviors of solutions, in which the diffusivity of sound $\delta$ heavily depends on the shear viscosity $\mu_{\mathrm{V}}$ and the bulk viscosity $\mu_{\mathrm{B}}$ of a given fluid (the viscous term in the conservation of momentum). We next illustrate the  significance of $\delta$ and introduce some  nomenclatures.
\begin{itemize}
	\item When $\delta>0$, it is called \emph{the viscous case} because the dissipative effect partly originates from the viscous term in the Navier-Stokes equation, and its energy exhibits exponentially stable on smooth bounded domains.
	\item When $\delta=0$, it is called \emph{the inviscid case} due to the vanishing viscous coefficient in the Euler equation from the modeling perspective, moreover, a corresponding energy of the equation is conserved.
	\item When $\delta<0$, it is called \emph{the chaotic case} since the solution is unstable in the sense of $\mathrm{e}^{ct}$ with $c>0$. The real parts of three eigenvalues are strictly positive.
\end{itemize}
 The Cauchy problem for \eqref{MGT-ALL} in the viscous case was studied by \cite{Pellicer-Said-2017} firstly, where some decay estimates of solutions for higher-dimensions have been derived by applying energy methods in the Fourier space. Later, the authors of \cite{Chen-Ikehata=2021} improved some results in the paper \cite{Pellicer-Said-2017} by considering the combined effect of dissipation and dispersion, which allows us to get optimal growth ($n=1,2$) and optimal decay ($n\geqslant3$) estimates. Moreover, some singular limit results with respect to the small thermal relaxation were obtained by applying energy methods associated with an integral unknown.

Let us turn to the JMGT equation \eqref{JMGT-ALL}, which has been initially studied by \cite{Kaltenbacher-Lasiecka-Pos-2012,Kaltenbacher-Nikolic-2019}. Due to the theme of this work, we just briefly introduce the progressive progress in the corresponding Cauchy problem for \eqref{JMGT-ALL} in the viscous case. The authors of \cite{Racke-Said-2020} demonstrated global (in time) well-posedness result in three-dimensions by higher-order energy methods together with a bootstrap argument, and derived some decay estimates of solutions. Quite recently, the author of \cite{Said-2021} improved the global (in time) existence result in \cite{Racke-Said-2020} for three-dimensions, where the smallness assumption for higher-order regular data was dropped in the framework of Sobolev spaces with negative index. So far the sharp estimates and asymptotic profiles of solutions are still open, especially, in the physical dimensions $n=1,2,3$.

Concerning other related works of the JMGT equation, we refer interested readers to \cite{Said-2021-Bes} for Besov spaces framework, \cite{Kaltenbacher-Nikolic-2021} for inviscid limits with respect to the diffusivity of sound, \cite{B-L-2020} for singular limits with respect to the thermal relaxation, \cite{Nikolic-Said-2021} for hereditary fluids, and \cite{Kaltenbacher-Nikolic-2022} for the time-fractional derivative models. Nevertheless, to the best of authors' knowledge, large-time asymptotic profiles for the solutions do not be deeply investigated, particularly, the second-order profiles (even for the linearized problem). We will  answer these questions in the present paper, and our novelty is to suggest a new connection between (Jordan-)MGT as well as diffusion-waves (sometimes Kuznetsov's equation) as $t\gg 1$. One may see Remarks \ref{Rem_MGT_Kuz} and \ref{Remark_Linear_Profi} for the linear problem; Remark \ref{Rem-4.n} for the nonlinear problem.

\subsection{Main purposes of the paper}
Our first purpose is to obtain sharp asymptotic profiles of solutions to the Cauchy problem for the following MGT equation (without loss of generality, we take $c_0^2=1$ hereinafter):
\begin{align}\label{Eq_MGT}
	\begin{cases}
		\tau \psi_{ttt}+\psi_{tt}-\Delta\psi-(\delta+\tau)\Delta\psi_t=0,&x\in\mb{R}^n,\ t>0,\\
		\psi(0,x)=\psi_0(x),\ \psi_t(0,x)=\psi_1(x),\ \psi_{tt}(0,x)=\psi_2(x),&x\in\mb{R}^n,
	\end{cases}
\end{align} 
with the thermal relaxation $\tau>0$ and the diffusivity of sound $\delta>0$. In Subsection \ref{Sub_Sec_Extend}, we extend the estimates in \cite[Section 2]{Chen-Ikehata=2021} to the non-vanishing data and some derivatives for the solution, which are the arrangements for exploring the nonlinear JMGT equation. Additionally, concerning large-time $t\gg1$, our main results in Section \ref{Sec_Prof} show the first-order profile
\begin{align*}
	\psi(t,x)\sim\frac{\sin(|D|t)}{|D|}\mathrm{e}^{\frac{\delta}{2}\Delta t}\Psi_{1,2}(x)
\end{align*}
with $\Psi_{1,2}(x):=\psi_1(x)+\tau\psi_2(x)$, and the new second-order profile
\begin{align*}
	\psi(t,x)-\frac{\sin(|D|t)}{|D|}\mathrm{e}^{\frac{\delta}{2}\Delta t}\Psi_{1,2}(x)\sim -\frac{\delta(4\tau-\delta)}{8}t\Delta\cos(|D|t)\mathrm{e}^{\frac{\delta}{2}\Delta t}\Psi_{1,2}(x)+\cos(|D|t)\mathrm{e}^{\frac{\delta}{2}\Delta t}\Psi_{0,2}(x)
\end{align*}
with $\Psi_{0,2}(x):=\psi_0(x)-\tau^2\psi_2(x)$, of solution to the MGT equation \eqref{Eq_MGT} for any $n\geqslant 1$. The main approaches are asymptotic analysis and WKB analysis, whose ideas are sharply differ from \cite{Chen-Ikehata=2021}. In Section \ref{Sec_Prof}, we not only derive optimal estimates for the solution (and its time-derivatives) in some $\dot{H}^k$ norms, but also prove the optimality for the error terms. It implies sharp asymptotic profiles that we obtained, where our methodology is provided at the beginning of Section \ref{Sec_Prof}. Our derived theorems throughout this part may resolve the relation between linearized Kuznetsov's equation and the MGT equation as $t\gg1$ (see the statement in Remark \ref{Remark_Linear_Profi}).

The further purpose of the paper is to study global (in time) qualitative properties of solution to the next nonlinear JMGT equation:
\begin{align}\label{JMGT_Dissipative}
	\begin{cases}
		\displaystyle{\tau \psi_{ttt}+\psi_{tt}-\Delta \psi-(\delta+\tau)\Delta \psi_t=\partial_t\left(\frac{B}{2A}(\psi_t)^2+|\nabla\psi|^2\right),}&x\in\mb{R}^n,\ t>0,\\
		\psi(0,x)=\psi_0(x),\ \psi_t(0,x)=\psi_1(x),\ \psi_{tt}(0,x)=\psi_2(x),&x\in\mb{R}^n,
	\end{cases}
\end{align}
with $\tau>0$, $\delta>0$ also, and $B/A>0$. Our major concerns are global (in time) existence of Sobolev solution for any $n\geqslant 1$, as well as asymptotic profiles for large-time. In Section \ref{Section_GESDS_JMGT}, basing on the sharp estimates for the linearized model \eqref{Eq_MGT}, and manufacturing relevant time-weighted Sobolev spaces, we conclude the existence and estimates of global (in time) small data Sobolev solution. The problematic of the proof is to close contraction estimates for lower dimensional case ($n=1,2$) because of the higher-order time-derivative $\psi_t\psi_{tt}$ in the nonlinear portion. To overcome this difficulty, we shall reformulate a new decomposition for nonlinearity. Especially, our proof is straightforward (some tools in Harmonic Analysis are utilized) so that we also can settle several closely related models in nonlinear acoustics (see Remark \ref{Rem_Other_Model}). Let us turn to the last goal, namely, asymptotic profiles of solutions to the JMGT equation \eqref{JMGT_Dissipative}. Setting down some profiles of diffusion-waves in Section \ref{Sec-4}, with the aid of Duhamel's principle proved in global (in time) existence part, we derive first- and second-order profiles of solutions to \eqref{JMGT_Dissipative}.

In order to complete this work, we derive singular limit with respect to the thermal relaxation $\tau$ for the MGT equation \eqref{Eq_MGT} in Section \ref{Section_Final_Remark} as a concluding remark. It provides a global (in time) convergence result in the $L^p$ framework (with $2\leqslant p\leqslant\infty$) on the acoustic velocity potential for any $n\geqslant 1$, which improved \cite[Theorem 4.1]{Chen-Ikehata=2021} by different approaches. 

Our main contributions in the present paper consist in obtaining first- and second-order large-time profiles of solutions (see Theorems \ref{Thm_First_Order_Prof} and \ref{Thm_Second_Order_Prof}), and in optimal estimates of solutions and error terms (see Theorem \ref{thm:2.3}) to the linear MGT equation; in proving global (in time) existence of Sobolev solution (see Theorem \ref{Thm_GESDS}), and in deriving approximations of global (in time) solution (see Theorems \ref{thm:4.1} and \ref{thm:4.2})  to the nonlinear JMGT equation in the viscous case.

\subsection{Basic notations}\label{subsection:Notation}
We define the following zones of the Fourier space:
\begin{align*}
\ml{Z}_{\intt}(\varepsilon_0):=\{|\xi|\leqslant\varepsilon_0\ll1\},\ \ 
\ml{Z}_{\bdd}(\varepsilon_0,N_0):=\{\varepsilon_0\leqslant |\xi|\leqslant N_0\},\ \ 
\ml{Z}_{\extt}(N_0):=\{|\xi|\geqslant N_0\gg1\}.
\end{align*}
 The cut-off functions $\chi_{\intt}(\xi),\chi_{\bdd}(\xi),\chi_{\extt}(\xi)\in \mathcal{C}^{\infty}$ equipping their supports in the corresponding zones $\ml{Z}_{\intt}(\varepsilon_0)$, $\ml{Z}_{\bdd}(\varepsilon_0/2,2N_0)$ and $\ml{Z}_{\extt}(N_0)$, respectively, so that $\chi_{\bdd}(\xi)=1-\chi_{\extt}(\xi)-\chi_{\intt}(\xi)$ for all $\xi \in \mb{R}^n$.

 The symbol $f\lesssim g$ means that there exists a positive constant $C$ fulfilling $f\leqslant Cg$, which may be changed on different lines, analogously, for $f\gtrsim g$. We denote $(a)^+:=\max\{a,0\}$ to be the non-negative part of $a\in\mb{R}$. The notation $\circ$ represents the inner product in $\mathbb{R}^{n}$.

   Moreover, $\dot{H}^s_q$ with $s\geqslant0$ and $1\leqslant q<\infty$, denote Riesz potential spaces based on $L^q$. Furthermore, $|D|^s$ with $s\geqslant0$ stands for the pseudo-differential operator with symbol $|\xi|^s$. We recall a weighted $L^1$ space such that $L_1^1:=\{ f\in L^1:|x|f\in L^1\}$,
and some related quantities for initial datum by
\begin{align*}
	M_j:=\int_{\mb{R}^n}\psi_j(x)\mathrm{d}x\ \ \mbox{as well as}\ \ P_j:=\int_{\mb{R}^n}x\psi_j(x)\mathrm{d}x,
\end{align*}
with $j=0,1,2$. Let us introduce Sobolev spaces endowing additionally $L^1$ regularity by 
\begin{align*}
	\ml{A}_s:=(H^{s+2}\cap L^1)\times(H^{s+1}\cap L^1)\times (H^{s}\cap L^1),
\end{align*}
 which will  be frequently used for some $s\geqslant0$.

Finally, we shall collect some time-dependent functions, which will be used in some estimates for solutions as follows:
   \begin{align*}
   	\ml{D}_n(t):=\begin{cases}
   		(1+t)^{\frac{1}{2}}&\mbox{if} \ \ n=1,\\
   		\sqrt{\ln (\mathrm{e}+t)}&\mbox{if}\ \ n=2,\\
   		(1+t)^{\frac{1}{2}-\frac{n}{4}}&\mbox{if} \ \ n\geqslant 3,
   	\end{cases}
\ \ \mbox{and} \ \
	\ml{D}_{n,s}(t):=\begin{cases}
		\ml{D}_{n}(t) &\mbox{if} \ \ s=0,\\
		(1+t)^{-\frac{s-1}{2}-\frac{n}{4}}&\mbox{if} \ \ s \geqslant  1.
	\end{cases}
   \end{align*}

\section{Asymptotic profiles and optimal estimates for the MGT equation}\label{Sec_Prof}
This section contributes to refined behaviors of solutions to the MGT equation \eqref{Eq_MGT}. Let us use the partial Fourier transform with respect to spatial variables  for the Cauchy problem \eqref{Eq_MGT} which yields
\begin{align}\label{Eq_MGT_Fourier}
	\begin{cases}
		\tau\widehat{\psi}_{ttt}+\widehat{\psi}_{tt}+(\delta+\tau)|\xi|^2\widehat{\psi}_t+|\xi|^2\widehat{\psi}=0,&\xi\in\mb{R}^n,\ t>0,\\
		\widehat{\psi}(0,\xi)=\widehat{\psi}_0(\xi),\ \widehat{\psi}_t(0,\xi)=\widehat{\psi}_1(\xi),\ \widehat{\psi}_{tt}(0,\xi)=\widehat{\psi}_2(\xi),&\xi\in\mb{R}^n.
	\end{cases}
\end{align}
The corresponding characteristic root $\lambda=\lambda(|\xi|)$ complies with
\begin{align}\label{Cubic_Eq}
	\tau\lambda^3+\lambda^2+(\delta+\tau)|\xi|^2\lambda+|\xi|^2=0.
\end{align}

Later, we will discuss higher-order asymptotic behaviors of the solution $\widehat{\psi}=\widehat{\psi}(t,\xi)$ to the model \eqref{Eq_MGT_Fourier}  for small frequencies $\xi\in\ml{Z}_{\intt}(\varepsilon_0)$ only. Concerning the other cases, according to the derived estimates in \cite{Pellicer-Said-2017,Chen-Ikehata=2021}, we claim that
\begin{align*}
	\big(1-\chi_{\intt}(\xi)\big)|\widehat{\psi}(t,\xi)|\lesssim \big(1-\chi_{\intt}(\xi)\big)\mathrm{e}^{-ct}\left(|\widehat{\psi}_0(\xi)|+|\xi|^{-1}|\widehat{\psi}_1(\xi)|+|\xi|^{-2}|\widehat{\psi}_2(\xi)|\right)
\end{align*}
with a positive constant $c>0$. The last inequality indicates exponential decay estimates holding for $\xi\in\ml{Z}_{\bdd}(\varepsilon_0,N_0)\cup\ml{Z}_{\extt}(N_0)$ with suitable regularities of initial data. Hence, the time-dependent coefficients of estimates (e.g. decay rate) for solutions heavily rely on small frequency part.
\medskip

\noindent\textbf{Methodology of Optimal Estimates:} In order to achieve our last aim of this section that is optimal estimates, namely, the same behaviors for upper \& lower bounds as $t\gg1$, for
\begin{align}\label{AIM}
	\|\partial_t^{\ell}\psi(t,\cdot)\|_{\dot{H}^k}\ \ \mbox{as well as}\ \ \|\partial_t^{\ell}\psi(t,\cdot)-\mbox{``first-order profiles''}\|_{\dot{H}^k}
\end{align}
with  $\ell=0,1,2$ and some $k\geqslant0$, our procedure will be separated into several steps.
\begin{description}
	\item[Step 1.] By refined WKB analysis for $|\xi|\ll 1$, three characteristic roots will be depicted intensively in Proposition \ref{Prop_Expansion_01}, specifically, even/odd powers of the factor $|\xi|$ will be clarified that improves \cite[Proposition 2.1]{Chen-Ikehata=2021}.
	\item[Step 2.] According to representation of solution and asymptotic behaviors in Step 1, we introduce two auxiliary functions $\widehat{\ml{G}}_{1,2}(t,\xi)$ in \eqref{auxilary_01} and \eqref{auxilary_02}, respectively. They bridge the relation between solution's kernels and diffusion-wave kernels in Propositions \ref{prop:2.4} as well as \ref{prop:2.5}. Basing on these bridges, first-/second-order asymptotic profiles will be concluded in Theorems \ref{Thm_First_Order_Prof} and \ref{Thm_Second_Order_Prof}.
	\item[Step 3a.] Motivated by the asymptotic profiles in Step 2, we will establish approximated functions and claim the approximations in Propositions \ref{prop:2.6} as well as \ref{prop:2.7}. By finely Fourier analysis for small frequencies, the optimal estimates for the constructed approximated functions will be achieved in Proposition \ref{prop:2.8}.
	\item[Step 3b.]  Finally, with the aid of the derived approximations in Step 3a, and derived optimal estimates for the approximated functions, we arrive at the optimal estimates of solutions in \eqref{AIM} due to the triangle inequality. Roughly speaking, we will employ
	\begin{align*}
		\|\Phi\|_{\dot{H}^k}\geqslant\underbrace{\|\chi_{\intt}(D)\Phi_{\mathrm{Prof}}\|_{\dot{H}^k}}_{\text{from Step 3a}}-\underbrace{\|\chi_{\intt}(D)(\Phi-\Phi_{\mathrm{Prof}})\|_{\dot{H}^k}}_{\text{from Step 2, faster decay}}\gtrsim\underbrace{\|\chi_{\intt}(D)\Phi_{\mathrm{Prof}}\|_{\dot{H}^k}}_{\text{from Step 3b}}\sim\mbox{our aim},
	\end{align*}
as $t\gg1$, in which the function $\Phi$ is our target, and $\Phi_{\mathrm{Prof}}$ denotes its corresponding profile obtained in Step 1.
\end{description}

\subsection{Preliminaries: Estimates of solutions with additionally $L^1$ initial data}\label{Sub_Sec_Extend}
Hereinafter, we will prepare several $L^2-L^2$ type estimates and $(L^2\cap L^1)-L^2$ type estimates as preliminaries for investigating qualitative properties of solutions to the nonlinear problem \eqref{JMGT_Dissipative} in the forthcoming sections.  According to the recent paper \cite{Chen-Ikehata=2021}, in the case of pairwise distinct eigenvalues $\lambda_j$ (definitely true for $|\xi|\leqslant\varepsilon\ll 1$), the solution to the Cauchy problem \eqref{Eq_MGT_Fourier} can be expressed via
\begin{align}\label{Representation_Fourier_u}
	\widehat{\psi} (t,\xi)=\widehat{K}_0(t,|\xi|)\widehat{\psi}_0 (\xi)+\widehat{K}_1(t,|\xi|)\widehat{\psi}_1 (\xi)+\widehat{K}_2(t,|\xi|)\widehat{\psi}_2 (\xi),
\end{align}
where the kernels in the Fourier space have the representations
\begin{align*}
	\widehat{K}_0(t,|\xi|)&:=\sum\limits_{j=1,2,3}\frac{\exp\big(\lambda_j(|\xi|)t\big)\prod_{k=1,2,3,\ k\neq j}\lambda_k(|\xi|)}{\prod_{k=1,2,3,\ k\neq j}\big(\lambda_j(|\xi|)-\lambda_k(|\xi|)\big)},\\
	\widehat{K}_1(t,|\xi|)&:=-\sum\limits_{j=1,2,3}\frac{\exp\big(\lambda_j(|\xi|)t\big)\sum_{k=1,2,3,\ k\neq j}\lambda_k(|\xi|)}{\prod_{k=1,2,3,\ k\neq j}\big(\lambda_j(|\xi|)-\lambda_k(|\xi|)\big)},\\
	\widehat{K}_2(t,|\xi|)&:=\sum\limits_{j=1,2,3}\frac{\exp\big(\lambda_j(|\xi|)t\big)}{\prod_{k=1,2,3,\ k\neq j}\big(\lambda_j(|\xi|)-\lambda_k(|\xi|)\big)},
\end{align*}
where $\lambda_j(|\xi|)$ solves the cubic \eqref{Cubic_Eq} for $j=1,2,3$. From the asymptotic expansions and stability analysis in the context of \cite{Pellicer-Said-2017,Chen-Ikehata=2021}, it is actually automatic that
\begin{itemize}
	\item $\lambda_{1}(|\xi|)=-\frac{1}{\tau}+\delta|\xi|^2+\ml{O}(|\xi|^3),\ \lambda_{2,3}(|\xi|)=\pm i|\xi|-\frac{\delta}{2}|\xi|^2+\ml{O}(|\xi|^3)$ if $|\xi|\to0$;
	\item $\lambda_1(|\xi|)=-\frac{1}{\delta+\tau}+\ml{O}(|\xi|^{-1}), \
	\lambda_{2,3}(|\xi|)=\pm i\sqrt{\frac{\delta+\tau}{\tau}}|\xi|-\frac{\delta}{2\tau(\delta+\tau)}+\ml{O}(|\xi|^{-1})$ if $|\xi|\to\infty$;
	\item $\Re\lambda_{j}(|\xi|)<0$ for $j=1,2,3$ if $|\xi|\not\to0$ and $|\xi|\not\to\infty$, i.e. bounded frequencies.
\end{itemize}
\begin{remark}
	The case for multiple eigenvalues, i.e. the zero-discriminant, appears only when $\xi\in\ml{Z}_{\bdd}(\varepsilon_0,N_0)$. Although the solution formula \eqref{Representation_Fourier_u} associated with the corresponding kernels is valid in the case of pairwise distinct eigenvalues, we will just use this formula \eqref{Representation_Fourier_u} for $\xi\in\ml{Z}_{\intt}(\varepsilon_0)\cup\ml{Z}_{\extt}(N_0)$ that three roots to \eqref{Cubic_Eq} are pairwise distinct (see more detail in \cite[Section 4]{Pellicer-Said-2017} or  \cite[Section 2.1]{Chen-Ikehata=2021}). Concerning the case for multiple eigenvalues, the exponential stability of solutions when $\delta>0$ is favorable for deriving some pointwise estimates when $\xi\in\ml{Z}_{\bdd}(\varepsilon_0,N_0)$.
\end{remark}

Let us analyze some asymptotic behaviors of the Fourier multipliers as $|\xi|\to0$ and $|\xi|\to\infty$, respectively. We may underline that $\widehat{K}_2(t,|\xi|)$ has been estimated in an appropriate way by  \cite[Section 2]{Chen-Ikehata=2021}. By direct computations, one derives
\begin{align*}
	\chi_{\intt}(\xi)|\partial_t^{\ell}\widehat{K}_0(t,|\xi|)|&\lesssim\chi_{\intt}(\xi)\left(|\xi|^{\ell}\mathrm{e}^{-c|\xi|^2t}+|\xi|^2\mathrm{e}^{-ct}\right)\ \ \mbox{if}\ \ \ell=0,1,2,\\
	\chi_{\intt}(\xi)|\partial_t^{\ell}\widehat{K}_1(t,|\xi|)|&\lesssim\begin{cases}
		\chi_{\intt}(\xi)\left(|\xi|^{-1}|\sin(|\xi|t)|\mathrm{e}^{-c|\xi|^2t}+|\xi|^2\mathrm{e}^{-ct}\right)&\mbox{if}\ \ \ell=0,\\
		\chi_{\intt}(\xi)\left(|\xi|^{\ell-1}\mathrm{e}^{-c|\xi|^2t}+|\xi|^2\mathrm{e}^{-ct}\right)&\mbox{if}\ \ \ell=1,2,
	\end{cases}\\
	\chi_{\intt}(\xi)|\partial_t^{\ell}\widehat{K}_2(t,|\xi|)|&\lesssim\begin{cases}
		\chi_{\intt}(\xi)\left(|\xi|^{-1}|\sin(|\xi|t)|\mathrm{e}^{-c|\xi|^2t}+\mathrm{e}^{-ct}\right)&\mbox{if}\ \ \ell=0,\\
		\chi_{\intt}(\xi)\left(|\xi|^{\ell-1}\mathrm{e}^{-c|\xi|^2t}+\mathrm{e}^{-ct}\right)&\mbox{if}\ \ \ell=1,2.
	\end{cases}
\end{align*}
For another, by the same way as before, in the case $|\xi|\to\infty$, we obtain the next pointwise estimates: 
\begin{align*}
	\chi_{\extt}(\xi)|\partial_t^{\ell}\widehat{K}_0(t,|\xi|)|&\lesssim\chi_{\extt}(\xi)\max\left\{1,|\xi|^{\ell-1}\right\}\mathrm{e}^{-ct},\\
	\chi_{\extt}(\xi)|\partial_t^{\ell}\widehat{K}_1(t,|\xi|)|&\lesssim\chi_{\extt}(\xi)|\xi|^{\ell-1}\mathrm{e}^{-ct},\\
	\chi_{\extt}(\xi)|\partial_t^{\ell}\widehat{K}_2(t,|\xi|)|&\lesssim\chi_{\extt}(\xi)|\xi|^{\ell-2}\mathrm{e}^{-ct},
\end{align*}
for all $\ell=0,1,2$. Clearly, thanks to negative real parts of eigenvalues in the bounded frequency zone, we are able to get
\begin{align*}
	\chi_{\bdd}(\xi)\left(|\partial_t^{\ell}\widehat{K}_0(t,|\xi|)|+|\partial_t^{\ell}\widehat{K}_1(t,|\xi|)|+|\partial_t^{\ell}\widehat{K}_2(t,|\xi|)|\right)\lesssim \chi_{\bdd}(\xi)\mathrm{e}^{-ct}
\end{align*}
with some positive constants $c>0$.

\begin{prop}\label{Prop_Estimate_Solution_itself}
	Let $(\psi_0 ,\psi_1 ,\psi_2 )\in\ml{A}_s$ for $s\in[0,\infty)$. Then, the solution to the Cauchy problem \eqref{Eq_MGT} fulfills the following estimates:
	\begin{align*}
		\|\psi (t,\cdot)\|_{L^2 }&\lesssim\ml{D}_{n}(t)\|(\psi_0 ,\psi_1 ,\psi_2 )\|_{(L^2 \cap L^1 )^3},\\
		\|\,|D|\psi(t,\cdot)\|_{L^2}&\lesssim (1+t)^{-\frac{n}{4}}\|(\psi_0 ,\psi_1 ,\psi_2 )\|_{(\dot{H}^1\cap L^1)\times(L^2 \cap L^1 )^2},\\
		\|\,|D|^{s+2}\psi (t,\cdot)\|_{L^2 }&\lesssim(1+t)^{-\frac{1}{2}}\|(\psi_0 ,\psi_1 ,\psi_2 )\|_{\dot{H}^{s+2}\times \dot{H}^{s+1} \times \dot{H}^{s} },\\
		\|\,|D|^{s+2}\psi (t,\cdot)\|_{L^2 }&\lesssim(1+t)^{-\frac{1}{2}-\frac{s}{2}-\frac{n}{4}}\|(\psi_0 ,\psi_1 ,\psi_2 )	\|_{(\dot{H}^{s+2} \cap L^1 )\times(\dot{H}^{s+1} \cap L^1 )\times (\dot{H}^{s} \cap L^1)},
	\end{align*}
	where the time-dependent coefficients were defined in Subsection \ref{subsection:Notation}.
\end{prop}
\begin{proof}
To prove our desired results, we will simply use H\"older's inequality and the Hausdorff-Young inequality (see, for instance, \cite[Proof of Theorem 1.1]{Ikehata-Natsume=2012}). Then, associated with the sharp estimates
\begin{align}\label{Eq_05}
	\|\chi_{\intt}(\xi)|\xi|^s\mathrm{e}^{-c|\xi|^2t}\|_{L^2}&\lesssim (1+t)^{-\frac{s}{2}-\frac{n}{4}},\\
	\|\chi_{\intt}(\xi)|\xi|^s\mathrm{e}^{-c|\xi|^2t}\|_{L^{\infty}}&\lesssim (1+t)^{-\frac{s}{2}},\notag\\
	\|\chi_{\intt}(\xi)|\xi|^{-1}|\sin(|\xi|t)|\mathrm{e}^{-c|\xi|^2t}\|_{L^2}&\lesssim\ml{D}_n(t),\notag
\end{align}
with $s\in[0,\infty)$, and exponential decay estimates for bounded and large frequencies, we follow the same approaches as the proofs of  \cite[Theorems 2.1 and 2.2]{Chen-Ikehata=2021} to complete the derivations of all estimates. Here, the requirements of regularity for initial datum come from the derived pointwise estimates in the Fourier space when $|\xi|\gg 1$.
\end{proof}

\begin{prop}\label{Prop_Estimate_Time_Derivative}
	Let $(\psi_0 ,\psi_1 ,\psi_2 )\in\ml{A}_s $ with $s\in[0,\infty)$.	Then, the time-derivatives of the solution to the Cauchy problem \eqref{Eq_MGT} fulfill the following estimates for $\ell=1,2$:
	\begin{align}
		\|\partial_t^{\ell}\psi (t,\cdot)\|_{L^2 }&\lesssim(1+t)^{\frac{1}{2}-\frac{\ell}{2}}\|(\psi_0 ,\psi_1 ,\psi_2 )\|_{(\dot{H}^{\ell-1} )^2\times L^2} ,\label{Est_05}\\
		\|\partial_t^{\ell}\psi (t,\cdot)\|_{L^2 }&\lesssim(1+t)^{\frac{1}{2}-\frac{\ell}{2}-\frac{n}{4}}\|(\psi_0 ,\psi_1 ,\psi_2 )\|_{(\dot{H}^{\ell-1}\cap L^1 )^2\times (L^2\cap L^1) },\label{Est_04}\\
		\||D| \partial_t \psi (t,\cdot)\|_{L^2 }&\lesssim(1+t)^{-\frac{1}{2}-\frac{n}{4}}\|(\psi_0 ,\psi_1 ,\psi_2 )\|_{(\dot{H}^{1}\cap L^1 )^2\times (L^2\cap L^1) },\label{eq:16}\\
		\|\,|D|^{s+2-\ell}\partial_t^{\ell}\psi (t,\cdot)\|_{L^2 }&\lesssim(1+t)^{-\frac{1}{2}}\|(\psi_0 ,\psi_1 ,\psi_2 )\|_{(\dot{H}^{s+1} )^2\times \dot{H}^{s} },\label{Est_09}\\
		\|\,|D|^{s+2-\ell}\partial_t^{\ell}\psi (t,\cdot)\|_{L^2 }&\lesssim(1+t)^{-\frac{1}{2}-\frac{s}{2}-\frac{n}{4}}\|(\psi_0 ,\psi_1 ,\psi_2 )\|_{(\dot{H}^{s+1}\cap L^1 )^2\times (\dot{H}^{s} \cap L^1) }.\label{Est_07}
	\end{align}
\end{prop}
\begin{proof}
	With the aid of the derived pointwise estimates, and the same philosophy as the proof of Proposition \ref{Prop_Estimate_Solution_itself}, one can complete the proof. Particularly, \eqref{eq:16} is a special case of \eqref{Est_07} with $s=0$ as well as $\ell=1$.
\end{proof}
By repeating the same procedure as Proposition \ref{Prop_Estimate_Time_Derivative} and using \eqref{Eq_05} under
\begin{align*}
|\partial_t^3\widehat{K}_2(t,|\xi|)|\lesssim\chi_{\intt}(\xi)|\xi|^2\mathrm{e}^{-c|\xi|^2t}+\big(1-\chi_{\intt}(\xi)\big)|\xi|\mathrm{e}^{-ct},
\end{align*} 
 we have the following corollary, which is essential to the study for one-dimensional nonlinear problem.
\begin{coro}\label{coro_Estimate_Time_Derivative}
	Let $\psi_2\in H^{s+1}\cap L^1$  with $s\in[0,\infty)$.
Then, the time-derivatives of the third kernel fulfill the following estimates for $\ell=0,1,2$:
	\begin{align}
		\|\,|D|^{s+2-\ell}(\partial_t^{\ell+1} K_2)(t,\cdot)\ast \psi_2(\cdot)  \|_{L^2 }&\lesssim(1+t)^{-1-\frac{s}{2}-\frac{n}{4}}\| \psi_2 \|_{\dot{H}^{s+1} \cap L^{1} },\label{eq:2.3}\\
		\|\,|D|^{s+2-\ell}(\partial_t^{\ell+1} K_2)(t,\cdot)\ast \psi_2(\cdot) \|_{L^2 }&\lesssim(1+t)^{-\frac{1}{2}}\| \psi_2\|_{\dot{H}^{s+1}}.\label{eq:2.4}
	\end{align}
\end{coro} 
\subsection{Asymptotic analysis on the characteristic roots for small frequencies}
Although in the last subsection we mentioned some pointwise estimates and asymptotic expansions for the eigenvalues as $|\xi|\to0$ (primarily examined by \cite[Section 2]{Chen-Ikehata=2021}), it is not sufficient to figure out second-order asymptotic profiles of solution. To  expand the eigenvalues for second-order, we will apply another scheme different from the one in \cite{Chen-Ikehata=2021}.

Let us take the positive number $\varepsilon_0$ to be a small constant so that
\begin{align}\label{discri}
	\mbox{discriminant of \eqref{Cubic_Eq}}<\left((\delta+\tau)(\delta+19\tau)-27\tau^2\right)|\xi|^4-4|\xi|^2\leqslant-2|\xi|^2<0,
\end{align}
when the frequency fulfills
\begin{align*}
	|\xi|\leqslant \varepsilon_0\ll 1,\ \ \mbox{e.g.}\ \ \varepsilon_0=1/\big((\delta+\tau)(\delta+19\tau)-27\tau^2\big).
\end{align*}
In other words, the cubic \eqref{Cubic_Eq} has one real root $\lambda_1=\lambda_1(|\xi|)$ as well as two (complex) conjugate roots $\lambda_{2,3}=\lambda_{2,3}(|\xi|)$ owning the form $\lambda_{2,3}=\mu_{\mathrm{R}}\pm i\mu_{\mathrm{I}}$ with $\mu_{\mathrm{R}}=\mu_{\mathrm{R}}(|\xi|)\in\mb{R}$ and $\mu_{\mathrm{I}}=\mu_{\mathrm{I}}(|\xi|)\in\mb{R}$. Absolutely, the roots fulfill
\begin{align}\label{Relation}
	\lambda_1+2\mu_{\mathrm{R}}=-\frac{1}{\tau},\ \ 2\lambda_1\mu_{\mathrm{R}}+\mu_{\mathrm{R}}^2+\mu_{\mathrm{I}}^2=\frac{\delta+\tau}{\tau}|\xi|^2\ \ \mbox{and} \ \ \lambda_1\left(\mu_{\mathrm{R}}^2+\mu_{\mathrm{I}}^2\right)=-\frac{1}{\tau}|\xi|^2.
\end{align}
For this reason, the asymptotic behavior of $\lambda_1$ plays a crucial role in determining $\mu_{\mathrm{R}}$ and $\mu_{\mathrm{I}}$.

By denoting 
\begin{align*}
A(|\xi|):=\frac{\delta+\tau}{\tau}|\xi|^2-\frac{1}{3\tau^2}\ \ \mbox{and}\ \ B(|\xi|):=\frac{2\tau-\delta}{3\tau^2}|\xi|^2+\frac{2}{27\tau^3},
\end{align*}
with $\triangle_{\mathrm{D}}(|\xi|):=\frac{1}{4}B(|\xi|)^2+\frac{1}{27}A(|\xi|)^3$, we apply Cardano's formula to get the real root (for small frequencies)
\begin{align*}
	\lambda_1(|\xi|)=-\frac{1}{3\tau}+\sqrt[3]{-\frac{1}{2}B(|\xi|)+\sqrt{\triangle_{\mathrm{D}}(|\xi|)}}+\sqrt[3]{-\frac{1}{2}B(|\xi|)-\sqrt{\triangle_{\mathrm{D}}(|\xi|)}}.
\end{align*}
Some straightforward computations imply
\begin{align*}
	\triangle_{\mathrm{D}}(|\xi|)=\frac{|\xi|^2}{27\tau^4}\left(1+\left(2\tau^2-5\tau\delta-\frac{1}{4}\delta^2\right)|\xi|^2+\tau(\delta+\tau)^3|\xi|^4\right)=:\frac{|\xi|^2}{27\tau^4}\left(1+g(|\xi|^2)\right).
\end{align*}
With the benefit of Taylor's expansion, we immediately have
\begin{align*}
	\sqrt{\triangle_{\mathrm{D}}(|\xi|)}&=\frac{|\xi|}{3\sqrt{3}\tau^2}\left(1+\sum\limits_{j=1}^{\infty}\frac{1}{j!}\prod\limits_{\ell=1}^j\left(\frac{3}{2}-\ell\right)g(|\xi|^2)^j\right)=:\frac{|\xi|}{3\sqrt{3}\tau^2}\left(1+\tilde{g}(|\xi|^2)\right).
\end{align*}
Analogously, we may represent
\begin{align*}
	\sqrt[3]{-\frac{1}{2}B(|\xi|)\pm\sqrt{\triangle_{\mathrm{D}}(|\xi|)}}&=-\frac{1}{3\tau}\sqrt[3]{1\mp3\sqrt{3}\tau|\xi|+\frac{9}{2}\tau(2\tau-\delta)|\xi|^2\mp3\sqrt{3}\tau|\xi|\tilde{g}(|\xi|^2)}\\
	&=:-\frac{1}{3\tau}\sqrt[3]{1+h_{\pm}(|\xi|)},
\end{align*}
packing
\begin{align*}
	\sqrt[3]{1+h_{\pm}(|\xi|)}=1+\frac{1}{3}h_{\pm}(|\xi|)-\frac{1}{9}h_{\pm}(|\xi|)^2+\sum\limits_{m=3}^{\infty}\frac{1}{m!}\prod_{\ell=1}^m\left(\frac{4}{3}-\ell\right)h_{\pm}(|\xi|)^m.
\end{align*}
By calculating
\begin{align*}
	\sum\limits_{\pm}\left(1+\frac{1}{3}h_{\pm}(|\xi|)-\frac{1}{9}h_{\pm}(|\xi|)^2\right)=2-3\tau\delta|\xi|^2+\ml{O}(|\xi|^3),
\end{align*}
it is an easy matter to show
\begin{align*}
	\lambda_1(|\xi|)=-\frac{1}{\tau}+\delta|\xi|^2+\ml{O}(|\xi|^4).
\end{align*}
Furthermore, according to the relation \eqref{Relation}, we find
\begin{align}\label{First-order}
	\mu_{\mathrm{R}}(|\xi|)=-\frac{\delta}{2}|\xi|^2+\ml{O}(|\xi|^4)\ \ \mbox{and}\ \ \mu_{\mathrm{I}}(|\xi|)=\pm |\xi|+\ml{O}(|\xi|^3).
\end{align}
The dominant parts of the last expansions are coincide with those in \cite[Section 2]{Chen-Ikehata=2021}.

In order to characterize asymptotic behaviors of $\lambda_1(|\xi|)$ for small frequencies, we may process a deep expansion
\begin{align*}
	\sum\limits_{\pm}\sqrt[3]{1+h_{\pm}(|\xi|)}=2+3\tau(2\tau-\delta)|\xi|^2+\sum\limits_{m=2}^{\infty}\frac{1}{m!}\prod\limits_{\ell=1}^m\left(\frac{4}{3}-\ell\right)\sum\limits_{\pm}h_{\pm}(|\xi|)^m.
\end{align*}
Let us now focus on the term (with $m\geqslant 2$)
\begin{align*}
	h_{\pm}(|\xi|)^m&=\left(\mp3\sqrt{3}\tau|\xi|+\frac{9}{2}\tau(2\tau-\delta)|\xi|^2\mp3\sqrt{3}\tau|\xi|\tilde{g}(|\xi|^2) \right)^m\\
	&=\sum\limits_{j=0}^m {m \choose j}\left(\mp 3\sqrt{3}\tau|\xi|\right)^{m-j}\sum\limits_{k=0}^j{j\choose k}\left(\frac{9}{2}\tau(2\tau-\delta)|\xi|^2\right)^{j-k}\left(\mp3\sqrt{3}\tau|\xi|\tilde{g}(|\xi|^2)\right)^k,
\end{align*}
whose key parts are $(\mp|\xi|)^{m-j}|\xi|^{2(j-k)}(\mp|\xi|^3)^k$, namely,
\begin{align*}
	(\mp 1)^{m-j+k}|\xi|^{m+j+k}\ \ \mbox{for}\ \ j=0,\dots,m\ \ \mbox{and}\ \ k=0,\dots,j.
\end{align*}
If $m+j+k$ is an odd number, the value of $m-j+k$ is odd too due to the trivial fact that $2(m+k)$ is even. For this reason, the sign of coefficients for $|\xi|^{\mathrm{odd}}$ is converse with $\mp$, which immediately leads to some cancellations after taking the sum. Thus, there only exist $|\xi|^{\mathrm{even}}$ in the formula of $\lambda_1(|\xi|)$ such that
\begin{align*}
	\lambda_1(|\xi|)=-\frac{1}{\tau}+\delta|\xi|^2+\tau\delta(\delta-\tau)|\xi|^4+\sum\limits_{j=3}^{\infty}\lambda_1^{(j)}|\xi|^{2j},
\end{align*}
where $\lambda_1^{(j)}\in\mb{R}$ for $j\geqslant 3$. From the relation \eqref{Relation} again, it allows us to obtain
\begin{align*}
	\mu_{\mathrm{R}}(|\xi|)=-\frac{\delta}{2}|\xi|^2-\frac{\tau\delta(\delta-\tau)}{2}|\xi|^4+\sum\limits_{j=3}^{\infty}\frac{\lambda_1^{(j)}}{2}|\xi|^{2j},
\end{align*}
and
\begin{align*}
	-\frac{1}{\tau}|\xi|^2&=\left(-\frac{1}{\tau}+\delta|\xi|^2+\tau\delta(\delta-\tau)|\xi|^4+\sum\limits_{j=3}^{\infty}\lambda_1^{(j)}|\xi|^{2j}\right)\\
	&\quad\,\, \times \left(\left(-\frac{\delta}{2}|\xi|^2-\frac{\tau\delta(\delta-\tau)}{2}|\xi|^4+\sum\limits_{j=3}^{\infty}\frac{\lambda_1^{(j)}}{2}|\xi|^{2j}\right)^2+\mu_{\mathrm{I}}(|\xi|)^2\right).
\end{align*}
Due to the situation that in the last equality all terms contain $|\xi|^{\mathrm{even}}$, basing on \eqref{First-order} we claim
\begin{align*}
	\mu_{\mathrm{I}}(|\xi|)= |\xi|+\frac{\delta(4\tau-\delta)}{8}|\xi|^3+\sum\limits_{j=3}^{\infty}\lambda_{2}^{(j)}|\xi|^{2j-1},
\end{align*} 
where $\lambda_{2}^{(j)}\in\mb{R}$ for $j\geqslant 3$ and it depends on $\lambda_1^{(m)}$ for $m\leqslant j$.

Summarizing the previous statements, we claim the following proposition. The cores are factors $|\xi|^3$, $|\xi|^4$ and some remainders appearing in the expansions.
\begin{prop}\label{Prop_Expansion_01}
	Let $\xi\in\ml{Z}_{\intt}(\varepsilon_0)$ with a small number $\varepsilon_0$. Then, the roots to the characteristic equation \eqref{Cubic_Eq} have the following asymptotic expansions:
	\begin{align*}
		\lambda_1(|\xi|)&=-\frac{1}{\tau}+\delta|\xi|^2+\tau\delta(\delta-\tau)|\xi|^4+\sum\limits_{j=3}^{\infty}\lambda_1^{(j)}|\xi|^{2j},\\
		\mu_{\mathrm{R}}(|\xi|)&=-\frac{\delta}{2}|\xi|^2-\frac{\tau\delta(\delta-\tau)}{2}|\xi|^4+\sum\limits_{j=3}^{\infty}\frac{\lambda_1^{(j)}}{2}|\xi|^{2j},\\
		\mu_{\mathrm{I}}(|\xi|)&= |\xi|+\frac{\delta(4\tau-\delta)}{8}|\xi|^3+\sum\limits_{j=3}^{\infty}\lambda_{2}^{(j)}|\xi|^{2j-1},
	\end{align*}
where $\lambda_1^{(j)}\in\mb{R}$ for $j\geqslant 3$, and $\lambda_{2}^{(j)}\in\mb{R}$ for $j\geqslant 3$ and it depends on $\lambda_1^{(m)}$ for $m\leqslant j$. Moreover, the second and third roots are $\lambda_{2,3}(|\xi|)=\mu_{\mathrm{R}}(|\xi|)\pm i\mu_{\mathrm{I}}(|\xi|)$.
\end{prop}

\subsection{Asymptotic expansions of solution for small frequencies}\label{Subsec_Solution_behavior}
Thanks to pairwise distinct characteristic roots for small frequencies (see also the negative discriminant \eqref{discri}), we may represent the solution by
\begin{align}
	\widehat{\psi}(t,\xi)&=\frac{-(\mu_{\mathrm{I}}^2+\mu_{\mathrm{R}}^2)\widehat{\psi}_0+2\mu_{\mathrm{R}}\widehat{\psi}_1-\widehat{\psi}_2}{2\mu_{\mathrm{R}}\lambda_1-\mu_{\mathrm{I}}^2-\mu_{\mathrm{R}}^2-\lambda_1^2}\mathrm{e}^{\lambda_1t}+\frac{(2\mu_{\mathrm{R}}\lambda_1-\lambda_1^2)\widehat{\psi}_0-2\mu_{\mathrm{R}}\widehat{\psi}_1+\widehat{\psi}_2}{2\mu_{\mathrm{R}}\lambda_1-\mu_{\mathrm{I}}^2-\mu_{\mathrm{R}}^2-\lambda_1^2}\cos(\mu_{\mathrm{I}}t)\mathrm{e}^{\mu_{\mathrm{R}}t}\notag\\
	&\quad+\frac{\lambda_1(\mu_{\mathrm{R}}\lambda_1+\mu_{\mathrm{I}}^2-\mu_{\mathrm{R}}^2)\widehat{\psi}_0+(\mu_{\mathrm{R}}^2-\mu_I^2-\lambda_1^2)\widehat{\psi}_1-(\mu_{\mathrm{R}}-\lambda_1)\widehat{\psi}_2}{\mu_{\mathrm{I}}(2\mu_{\mathrm{R}}\lambda_1-\mu_{\mathrm{I}}^2-\mu_{\mathrm{R}}^2-\lambda_1^2)}\sin(\mu_{\mathrm{I}}t)\mathrm{e}^{\mu_{\mathrm{R}}t}.\label{Represent_01}
\end{align}
For the sake of concreteness and convenience, we take the next configurations by getting rid of higher-order remainders to be two bridges for sharp asymptotic profiles:
\begin{align}
	\widehat{\ml{G}}_1(t,\xi)&:=
	\frac{-\lambda_1^2 \widehat{\psi}_1+\lambda_1\widehat{\psi}_2}{\mu_{\mathrm{I}}(2\mu_{\mathrm{R}}\lambda_1-\mu_{\mathrm{I}}^2-\mu_{\mathrm{R}}^2-\lambda_1^2)}\sin(\mu_{\mathrm{I}}t)\mathrm{e}^{\mu_{\mathrm{R}}t},\label{auxilary_01}\\
	\widehat{\ml{G}}_2(t,\xi)&:=
	\frac{-\lambda_1^2 \widehat{\psi}_0+\widehat{\psi}_2}{
	2\mu_{\mathrm{R}}\lambda_1-\mu_{\mathrm{I}}^2-\mu_{\mathrm{R}}^2-\lambda_1^2}\cos(\mu_{\mathrm{I}}t)\mathrm{e}^{\mu_{\mathrm{R}}t
	}\label{auxilary_02}.
\end{align}
Then, direct subtractions lead to 
\begin{align}
	\left| \partial_{t}^{\ell} \left( \widehat{\psi}(t,\xi)-\widehat{\ml{G}}_1(t,\xi) \right)\right| &\lesssim \mathrm{e}^{-c|\xi|^{2}t} |\xi|^{\ell} \left(|\widehat{\psi}_0(\xi)|+|\widehat{\psi}_1(\xi)|+|\widehat{\psi}_2(\xi)|\right),\label{eq:2.-1}\\ 
	\left|\partial_{t}^{\ell} \left( \widehat{\psi}(t,\xi)-\widehat{\ml{G}}_1(t,\xi) -\widehat{\ml{G}}_2(t,\xi) \right) \right| &\lesssim \mathrm{e}^{-c|\xi|^{2}t} |\xi|^{\ell+1} \left(|\widehat{\psi}_0(\xi)|+|\widehat{\psi}_1(\xi)|+|\widehat{\psi}_2(\xi)|\right),\label{eq:2.0}
\end{align}
for $\ell=0,\dots,3$ and $\xi\in\ml{Z}_{\intt}(\varepsilon_0)$ by Proposition \ref{Prop_Expansion_01}. What's more, we also introduce the Fourier multipliers related to diffusion-waves as follows: 
\begin{align}\label{Symbol_J}
	\widehat{\ml{J}}_0(t,|\xi|):=\frac{\sin(|\xi|t)}{|\xi|}\mathrm{e}^{-\frac{\delta}{2}|\xi|^2t}\ \ \mbox{as well as}\ \  \widehat{\ml{J}}_1(t,|\xi|):=\cos(|\xi|t)\,\mathrm{e}^{-\frac{\delta}{2}|\xi|^{2} t},
\end{align}
and the combinations of initial datum $\widehat{\Psi}_{0,2}(\xi):=\widehat{\psi}_1(\xi)-\tau^{2}\widehat{\psi}_2(\xi)$ and
$\widehat{\Psi}_{1,2}(\xi):=\widehat{\psi}_1(\xi)+\tau\widehat{\psi}_2(\xi)$, which are crucial to describe the asymptotic profiles of the solution $\psi(t,x)$ for the linearized problem.

As we did in \eqref{eq:2.-1} and \eqref{eq:2.0}, we have derived error estimates between solutions and the configurations. So, our next propositions will devote to the error estimates between the configurations $\widehat{\ml{G}}_{1,2}(t,\xi)$ and Fourier multipliers for diffusion-waves:
\begin{prop} \label{prop:2.4}
	Let $\xi\in\ml{Z}_{\intt}(\varepsilon_0)$ with a small number $\varepsilon_0$. Then, the differences between Fourier multipliers satisfy
	\begin{align}
		 |\widehat{\ml{G}}_1(t,\xi)-\widehat{\ml{J}}_0(t,|\xi|)\widehat{\Psi}_{1,2}(\xi)| &\lesssim \mathrm{e}^{-c|\xi|^{2}t} \left(|\widehat{\psi}_1(\xi)|+|\widehat{\psi}_2(\xi)|\right), \label{eq:2.44}\\
	 |\widehat{\ml{G}}_2(t,\xi)-\widehat{\ml{J}}_1(t,|\xi|)\widehat{\Psi}_{0,2}(\xi)|&\lesssim |\xi|\mathrm{e}^{-c|\xi|^{2}t}  \left(|\widehat{\psi}_0(\xi)|+|\widehat{\psi}_2(\xi)|\right), \label{eq:2.45} 
	\end{align}
as well as
\begin{align}
|\widehat{\ml{G}}_1(t,\xi)-\widehat{\ml{J}}_0(t,|\xi|)\widehat{\Psi}_{1,2}(\xi)-\widehat{\ml{N}}_0(t,|\xi|)\widehat{\Psi}_{1,2}(\xi)|  \lesssim |\xi|\mathrm{e}^{-c|\xi|^{2}t}  \left(|\widehat{\psi}_1(\xi)|+|\widehat{\psi}_2(\xi)|\right), \label{eq:2.46} 
\end{align}
for any $n\geqslant 1$, where the Fourier multiplier is denoted by 
\begin{align*}
	\widehat{\ml{N}}_0(t,|\xi|):=-t \frac{\delta(4\tau-\delta)}{8}|\xi|^{2} \widehat{\ml{J}}_1(t,|\xi|).
\end{align*}
\end{prop}
\begin{proof}
We begin with the proof of \eqref{eq:2.44}.
We decompose the target into four parts:
\begin{align}\label{eq:2.47}
	\mbox{LHS of \eqref{eq:2.44}}=D_1(t,\xi)+D_2(t,\xi)+D_3(t,\xi)+D_4(t,\xi), 
\end{align}
where
\begin{align*}
	D_1(t,\xi)& := \left(
	\mathrm{e}^{\mu_{\mathrm{R}}t}-\mathrm{e}^{-\frac{\delta}{2}|\xi|^2t} \right) \frac{-\lambda_1^2 \widehat{\psi}_1+\lambda_1\widehat{\psi}_2}{\mu_{\mathrm{I}}(2\mu_{\mathrm{R}}\lambda_1-\mu_{\mathrm{I}}^2-\mu_{\mathrm{R}}^2-\lambda_1^2)}\sin(\mu_{\mathrm{I}}t), \\
	D_2(t,\xi)& := \mathrm{e}^{-\frac{\delta}{2}|\xi|^2t} \big( \sin(\mu_{\mathrm{I}}t)-\sin(|\xi|t) \big)\frac{-\lambda_1^2 \widehat{\psi}_1+\lambda_1\widehat{\psi}_2}{\mu_{\mathrm{I}}(2\mu_{\mathrm{R}}\lambda_1-\mu_{\mathrm{I}}^2-\mu_{\mathrm{R}}^2-\lambda_1^2)}, \\
	D_3(t,\xi)& := \mathrm{e}^{-\frac{\delta}{2}|\xi|^2t} \sin(|\xi|t) \left( \frac{1}{\mu_{\mathrm{I}}(2\mu_{\mathrm{R}}\lambda_1-\mu_{\mathrm{I}}^2-\mu_{\mathrm{R}}^2-\lambda_1^2)}+\frac{\tau^{2}}{|\xi|} \right)(-\lambda_1^2 \widehat{\psi}_1+\lambda_1\widehat{\psi}_2) , \\
	D_4(t,\xi)& := \mathrm{e}^{-\frac{\delta}{2}|\xi|^2t} \frac{\sin(|\xi|t)}{|\xi|} 
	\left( 
	(\tau^{2} \lambda_1^2-1) \widehat{\psi}_1-(\tau^{2} \lambda_1+\tau  \widehat{\psi}_2)
	\right).
\end{align*}
Using the trick 
\begin{align*}
\chi_{\intt}(\xi)\left|\mathrm{e}^{\mu_{\mathrm{R}}t}-\mathrm{e}^{-\frac{\delta}{2}|\xi|^2t}\right|&\lesssim\chi_{\intt}(\xi)|\xi|^4t\mathrm{e}^{-\frac{\delta}{2}|\xi|^2t}\left|\int_0^1\mathrm{e}^{-\frac{\tau\delta(\delta-\tau)}{2}|\xi|^4t\eta+\ml{O}(|\xi|^6)t\eta}\mathrm{d}\eta \right|\notag\\
&\lesssim\chi_{\intt}(\xi)|\xi|^2\mathrm{e}^{-c|\xi|^2t}
\end{align*}
and Proposition \ref{Prop_Expansion_01},
we have 
\begin{align} \label{eq:2.5}
\chi_{\intt}(\xi)|D_1(t,\xi)| \lesssim\chi_{\intt}(\xi)|\xi| \mathrm{e}^{-c|\xi|^2t}\left(|\widehat{\psi}_1(\xi)|+|\widehat{\psi}_2(\xi)|\right). 
\end{align}
Motivated by
\begin{align*}
	\chi_{\intt}(\xi)|\sin(\mu_{\mathrm{I}}t)-\sin(|\xi|t)|&\lesssim\chi_{\intt}(\xi)|\xi|^3t\left|\cos\left(\tilde{\mu}_{\mathrm{I}}t\right)\right|\lesssim\chi_{\intt}(\xi)|\xi|^3t
\end{align*}
from the mean value theorem with $\tilde{\mu}_{\mathrm{I}}\in(|\xi|,\mu_{\mathrm{I}})$, one realizes that
\begin{align} \label{eq:2.6}
	\chi_{\intt}(\xi)|D_2(t,\xi)| \lesssim\chi_{\intt}(\xi)\mathrm{e}^{-c|\xi|^2t}\left(|\widehat{\psi}_1(\xi)|+|\widehat{\psi}_2(\xi)|\right). 
\end{align}
Moreover, due to the estimates 
\begin{align*}
	\chi_{\intt}(\xi)\left|\frac{1}{\mu_{\mathrm{I}}(2\mu_{\mathrm{R}}\lambda_1-\mu_{\mathrm{I}}^2-\mu_{\mathrm{R}}^2-\lambda_1^2)}+\frac{\tau^2}{|\xi|}\right|&\lesssim\chi_{\intt}(\xi)\frac{|\,|\xi|-\tau^2\lambda_1^2\mu_{\mathrm{I}}|+\tau^2\mu_{\mathrm{I}}|2\mu_{\mathrm{R}}\lambda_1-\mu_{\mathrm{I}}^2-\mu_{\mathrm{R}}^2|}{\mu_{\mathrm{I}}|\xi|\,|2\mu_{\mathrm{R}}\lambda_1-\mu_{\mathrm{I}}^2-\mu_{\mathrm{R}}^2-\lambda_1^2|}\\
	&\lesssim\chi_{\intt}(\xi)|\xi|,
\end{align*}
we obtain 
\begin{align}\label{eq:2.7}
\chi_{\intt}(\xi)\big(|D_3(t,\xi)|+|D_4(t,\xi)|\big) \lesssim\chi_{\intt}(\xi) |\xi| \mathrm{e}^{-c|\xi|^2t}\left(|\widehat{\psi}_1(\xi)|+|\widehat{\psi}_2(\xi)|\right). 
\end{align}
Summing up the estimates \eqref{eq:2.47}, \eqref{eq:2.5}, \eqref{eq:2.6} and \eqref{eq:2.7}, we have the desired estimate \eqref{eq:2.44}. Another estimate \eqref{eq:2.45} can be demonstrated by a parallel way.

Next, we show the estimate \eqref{eq:2.46}, whose problematic part is the $D_2(t,\xi)$ because the lack of decay factor $|\xi|$ in the derived estimate \eqref{eq:2.6}. In other words, we have to search the dominant component in $D_2(t,\xi)$.
In the companion work, we just modify the decomposition as follows:
\begin{align*}
	\mbox{LHS of \eqref{eq:2.47}}=D_1(t,\xi)+D_{2,1}(t,\xi)+D_{2,2}(t,\xi)+D_3(t,\xi)+D_4(t,\xi)
\end{align*}
where $D_j(t,\xi)$ with $j=1,3,4$ were defined in preceding part of the text, and
\begin{align*}
	D_{2,1}(t,\xi)& := \mathrm{e}^{-\frac{\delta}{2}|\xi|^2t} \left( \sin(\mu_{\mathrm{I}}t)-\sin(|\xi|t)-\frac{\delta(4 \tau-\delta)}{8}|\xi|^{3} t \cos(|\xi| t) \right)\frac{-\lambda_1^2 \widehat{\psi}_1+\lambda_1\widehat{\psi}_2}{\mu_{\mathrm{I}}(2\mu_{\mathrm{R}}\lambda_1-\mu_{\mathrm{I}}^2-\mu_{\mathrm{R}}^2-\lambda_1^2)}, \\
	D_{2,2}(t,\xi)& := \mathrm{e}^{-\frac{\delta}{2}|\xi|^2t} \frac{\delta(4 \tau-\delta)}{8} |\xi|^{3} t \cos(|\xi| t) 
	\left(
	\frac{-\lambda_1^2 \widehat{\psi}_1+\lambda_1\widehat{\psi}_2}{\mu_{\mathrm{I}}(2\mu_{\mathrm{R}}\lambda_1-\mu_{\mathrm{I}}^2-\mu_{\mathrm{R}}^2-\lambda_1^2)} 
	-
	\frac{\widehat{\psi}_1+\tau \widehat{\psi}_2}{|\xi|}
	\right). 
\end{align*}
Noting the expansion 
\begin{align*}
	\sin(\mu_{\mathrm{I}}t)=\sin(|\xi|t) +(\mu_{\mathrm{I}}-|\xi|)t \cos(|\xi| t) -\frac{(\mu_{\mathrm{I}}-|\xi|)^{2}}{2}t^{2} \sin (\bar{\mu}_{\mathrm{I}} t)
\end{align*}
with $\bar{\mu}_{\mathrm{I}}\in(|\xi|,\mu_{\mathrm{I}})$,
which can be rephrased via 
\begin{align*}
	\sin(\mu_{\mathrm{I}}t)-\sin(|\xi|t) -\frac{\delta(4 \tau-\delta)}{8} |\xi|^{3} t \cos(|\xi| t) = t \ml{O}(|\xi|^5)+t^{2} \ml{O}(|\xi|^6)
\end{align*}
as $|\xi| \to 0$, by employing Proposition \ref{Prop_Expansion_01} again,
it follows \eqref{eq:2.46} as in the proof of \eqref{eq:2.44} since
\begin{align*}
	\chi_{\intt}(\xi)\left|\frac{-\lambda_1^2}{\mu_{\mathrm{I}}(2\mu_{\mathrm{R}}\lambda_1-\mu_{\mathrm{I}}^2-\mu_{\mathrm{R}}^2-\lambda_1^2)}-\frac{1}{|\xi|}\right|+\chi_{\intt}(\xi)\left|\frac{\lambda_1}{\mu_{\mathrm{I}}(2\mu_{\mathrm{R}}\lambda_1-\mu_{\mathrm{I}}^2-\mu_{\mathrm{R}}^2-\lambda_1^2)}-\frac{\tau}{|\xi|}\right|\lesssim\chi_{\intt}(\xi)|\xi|.
\end{align*}
Consequently, we complete the proof of this proposition. 
\end{proof}
With the same observation as 
 $\partial_{t}^{\ell} \widehat{\ml{G}}_1(t,\xi)$ for $\ell=1,2$ and small $|\xi|$, 
 one may see
 \begin{align}\label{eq:2.81}
& \partial_{t} \widehat{\ml{G}}_1(t,\xi) \sim \frac{-\lambda_1^2 \widehat{\psi}_1+\lambda_1\widehat{\psi}_2}{2\mu_{\mathrm{R}}\lambda_1-\mu_{\mathrm{I}}^2-\mu_{\mathrm{R}}^2-\lambda_1^2}\cos(\mu_{\mathrm{I}}t)\mathrm{e}^{\mu_{\mathrm{R}}t}
\sim \widehat{\ml{J}}_1(t,|\xi|)\widehat{\Psi}_{1,2}(\xi),\\
& \partial_{t}^{2} \widehat{\ml{G}}_1(t,\xi) \sim -\frac{\mu_{\mathrm{I}}(-\lambda_1^2 \widehat{\psi}_1+\lambda_1\widehat{\psi}_2)}{2\mu_{\mathrm{R}}\lambda_1-\mu_{\mathrm{I}}^2-\mu_{\mathrm{R}}^2-\lambda_1^2}\sin(\mu_{\mathrm{I}}t)\mathrm{e}^{\mu_{\mathrm{R}}t}
\sim - \widehat{\ml{J}}_0(t,|\xi|) |\xi|^{2} \widehat{\Psi}_{1,2}(\xi), 
 \end{align}
to be the first-order asymptotic profiles, and 
\begin{align*}
\partial_{t} \widehat{\ml{G}}_1(t,\xi) - \widehat{\ml{J}}_1(t,|\xi|)\widehat{\Psi}_{1,2}(\xi)
& \sim-\delta\left(\frac{4\tau-\delta}{8}|\xi|^2t+\frac{1}{2}\right)|\xi|^2\widehat{\ml{J}}_0(t,|\xi|)\widehat{\Psi}_{1,2}(\xi)=:\widehat{\ml{N}}_1(t,|\xi|)\widehat{\Psi}_{1,2}(\xi),\\
\partial_{t}^{2} \widehat{\ml{G}}_1(t,\xi)+\widehat{\ml{J}}_0(t,|\xi|)|\xi|^{2}  \widehat{\Psi}_{1,2}(\xi) 
&\sim -\delta\left(\frac{4\tau-\delta}{8}|\xi|^2t+1\right)|\xi|^2\widehat{\ml{J}}_1(t,|\xi|)\widehat{\Psi}_{1,2}(\xi)=:\widehat{\ml{N}}_2(t,|\xi|)\widehat{\Psi}_{1,2}(\xi),
\end{align*}
to be the second-order asymptotic profiles.
Similarly, we also arrive at
\begin{align}
&\partial_{t} \widehat{\ml{G}}_2(t,\xi) 
\sim - \widehat{\ml{J}}_0(t,|\xi|)|\xi|^{2}\widehat{\Psi}_{0,2}(\xi),\\
&\partial_{t}^{2} \widehat{\ml{G}}_2(t,\xi) 
\sim - \widehat{\ml{J}}_1(t,|\xi|)|\xi|^{2}\widehat{\Psi}_{0,2}(\xi),\label{eq:2.83}
\end{align}
when $|\xi|\leqslant \varepsilon_0\ll 1$. Summarizing the last computations, we have the pointwise estimates of the fundamental solutions in the Fourier space. 
\begin{prop} \label{prop:2.5}
Let $\xi\in\ml{Z}_{\intt}(\varepsilon_0)$ with a small number $\varepsilon_0$.
Then, the differences between Fourier multipliers fulfill
\begin{align*}
| \partial_{t} \widehat{\ml{G}}_1(t,\xi) - \widehat{\ml{J}}_1(t,|\xi|)\widehat{\Psi}_{1,2}(\xi)  |
&\lesssim |\xi| \mathrm{e}^{-c|\xi|^2t}\left(|\widehat{\psi}_1(\xi)|+|\widehat{\psi}_2(\xi)|\right),\\
| \partial_{t}^{2} \widehat{\ml{G}}_1(t,\xi) + \widehat{\ml{J}}_0(t,|\xi|)|\xi|^{2} \widehat{\Psi}_{1,2}(\xi)  |
&\lesssim |\xi|^{2} \mathrm{e}^{-c|\xi|^2t}\left(|\widehat{\psi}_1(\xi)|+|\widehat{\psi}_2(\xi)|\right),\\
| \partial_{t} \widehat{\ml{G}}_2(t,\xi) + \widehat{\ml{J}}_0(t,|\xi|)|\xi|^{2}\widehat{\Psi}_{0,2}(\xi)  |
&\lesssim |\xi|^{2} \mathrm{e}^{-c|\xi|^2t}\left(|\widehat{\psi}_0(\xi)|+|\widehat{\psi}_2(\xi)|\right),\\
| \partial_{t}^{2} \widehat{\ml{G}}_2(t,\xi) + \widehat{\ml{J}}_1(t,|\xi|)|\xi|^{2} \widehat{\Psi}_{0,2}(\xi)  |
&\lesssim |\xi|^{3} \mathrm{e}^{-c|\xi|^2t}\left(|\widehat{\psi}_0(\xi)|+|\widehat{\psi}_2(\xi)|\right), 
\end{align*}
and
\begin{align*}
		 | \partial_{t} \widehat{\ml{G}}_1(t,\xi) - \widehat{\ml{J}}_1(t,|\xi|)\widehat{\Psi}_{1,2}(\xi) -\widehat{\ml{N}}_1(t,|\xi|) \widehat{\Psi}_{1,2}(\xi)| &\lesssim |\xi|^{2} \mathrm{e}^{-c|\xi|^2t}\left(|\widehat{\psi}_1(\xi)|+|\widehat{\psi}_2(\xi)|\right),\\
	| \partial_{t}^{2} \widehat{\ml{G}}_1(t,\xi) + \widehat{\ml{J}}_0(t,|\xi|)|\xi|^{2}\widehat{\Psi}_{1,2}(\xi) -\widehat{\ml{N}}_2(t,|\xi|)\widehat{\Psi}_{1,2}(\xi)| &\lesssim |\xi|^{3} \mathrm{e}^{-c|\xi|^2t}\left(|\widehat{\psi}_1(\xi)|+|\widehat{\psi}_2(\xi)|\right).
\end{align*}
\end{prop}

\subsection{Large-time profiles of first- and second-order for the MGT equation}\label{Subsec_first-second}
Let us state our first theorem on first-order asymptotic profiles of solution to the MGT equation \eqref{Eq_MGT} for large-time. This result simplifies the complicated profile $\tau J(t,|D|)\psi_2(x)$ in \cite[Equation (22)]{Chen-Ikehata=2021}, and indicates essential structure of the viscous MGT equation.
\begin{theorem}\label{Thm_First_Order_Prof}
Let $(\psi_0 ,\psi_1 ,\psi_2 )\in \ml{A}_s$ for $s\in[0,\infty)$. 
Then, the solution to the Cauchy problem \eqref{Eq_MGT} fulfills the following refined estimates:
\begin{align} \label{eq:2.9}
	\left\|\partial_t^{\ell}\psi(t,\cdot)-\frac{1}{|D|}\mathrm{e}^{\frac{\delta}{2}\Delta t}\partial_t^{\ell}\sin(|D|t)\Psi_{1,2}(\cdot)\right\|_{\dot{H}^k}\lesssim(1+t)^{-\frac{n}{4}-\frac{k+\ell}{2}}\|(\psi_0 ,\psi_1 ,\psi_2 )\|_{\ml{A}_{(k-2+\ell)^+}}
\end{align}
for $0\leqslant k\leqslant s+2-\ell$ with $\ell=0,1,2$, where the datum space $\ml{A}_s$ was defined in Subsection \ref{subsection:Notation}.
\end{theorem}

\begin{remark}
	By subtracting the first-order profiles in Theorem \ref{Thm_First_Order_Prof}, we observe an improvement of decay rate $t^{-\frac{3}{4}}$ when $n=1$, $(t\ln t)^{-\frac{1}{2}}$ when $n=2$, and $t^{-\frac{1}{2}}$ when $n\geqslant 3$ of the time-dependent coefficients for initial datum as $t\gg1$ if we compare with the first estimate in Proposition \ref{Prop_Estimate_Solution_itself}. For the other case, the decay rate has been enhanced $t^{-\frac{1}{2}}$ for large-time.
\end{remark}

\begin{remark}\label{Rem_MGT_Kuz}
By employing \cite[Lemma 2.1]{Ikehata=2014}, the solution $\varphi=\varphi(t,x)$ to linear Kuznetsov's equation with vanishing first data, namely,
\begin{align}\label{Lienar_Kuznetsov}
	\begin{cases}
		\varphi_{tt}-\Delta\varphi-\delta\Delta\varphi_t=0,&x\in\mb{R}^n,\ t>0,\\
		\varphi(0,x)=0,\ \varphi_t(0,x)=\Psi_{1,2}(x),&x\in\mb{R}^n,
	\end{cases}
\end{align}
fulfilling the refined estimate as $t\gg1$ that
\begin{align}\label{Eq_02}
\left\|\chi_{\intt}(D)\big(\varphi(t,\cdot)-\ml{J}_0(t,|D|)\Psi_{1,2}(\cdot)\big)\right\|_{L^2}\lesssim t^{-\frac{n}{4}}\|(\psi_1,\psi_2)\|_{(L^1)^2}.
\end{align}
Then, using the triangle inequality for \eqref{eq:2.9} with $k=0$, $\ell=0$ as well as \eqref{Eq_02}, we claim
\begin{align*}
	\left\|\psi(t,\cdot)-\varphi(t,\cdot)\right\|_{L^2}\lesssim t^{-\frac{n}{4}}\|(\psi_0,\psi_1,\psi_2)\|_{(L^2\cap L^1)^3}
\end{align*}
as $t\gg1$ for any $n\geqslant 1$. It shows approximated relation (or generalized diffusion phenomenon) between the MGT equation \eqref{Eq_MGT} and the linearized Kuznetsov's equation \eqref{Lienar_Kuznetsov} under some data conditions. This result completely answers the question proposed in \cite[Remark 3.2]{Chen-Ikehata=2021}.
\end{remark}
\begin{proof}
If one employs triangle inequalities between \eqref{eq:2.-1}, \eqref{eq:2.44} and those in Proposition \ref{prop:2.5}, then we obtain our desired difference with $\ell=0,1,2$, which can be controlled by
\begin{align*}
	\chi_{\intt}(\xi)|\xi|^{\ell}\mathrm{e}^{-c|\xi|^2t}\left(|\widehat{\psi}_0(\xi)|+|\widehat{\psi}_1(\xi)|+|\widehat{\psi}_2(\xi)|\right).
\end{align*}
With the aid of \eqref{Eq_05}, we complete the proof.
\end{proof}

Next, we state the second-order profiles, which leads to faster decay estimates than the subtraction of the first-order profiles in Theorem \ref{Thm_First_Order_Prof}. 
%
\begin{theorem}\label{Thm_Second_Order_Prof}
Under the same assumption as those in Theorem \ref{Thm_First_Order_Prof}, the refined estimates hold
\begin{align}
	&\left\|\partial_t^{\ell}\psi(t,\cdot)-\left(\frac{1}{|D|}\mathrm{e}^{\frac{\delta}{2}\Delta t}\partial_t^{\ell}\sin(|D|t)+\ml{N}_{\ell}(t,|D|)\right)\Psi_{1,2}(\cdot)-\mathrm{e}^{\frac{\delta}{2}\Delta t}\partial_t^{\ell}\cos(|D|t)\Psi_{0,2}(\cdot)\right\|_{\dot{H}^k}\notag\\
	&\qquad\lesssim (1+t)^{-\frac{n}{4}-\frac{k+1+\ell}{2}}\|(\psi_0 ,\psi_1 ,\psi_2 )\|_{\ml{A}_{(k-2+\ell)^+}}\label{eq:2.12}
\end{align}
for $0\leqslant k\leqslant s+2-\ell$ with $\ell=0,1,2$.
\end{theorem}
%
\begin{proof}
Here, we only prove the estimate \eqref{eq:2.12} with $\ell=0$. 
In what follows, we just need to concentrate on the case $\xi\in\ml{Z}_{\intt}(\varepsilon_0)$ by virtue of the other cases yielding some exponential decay estimates. Our proof strongly relies on the new expansions in Proposition \ref{Prop_Expansion_01}, especially, the odd-order powers of $|\xi|$ in higher-order expansion of $\mu_{\mathrm{I}}(|\xi|)$ play a crucial role. From the solution's formula \eqref{Represent_01}, \eqref{eq:2.0} with $\ell=0$, \eqref{eq:2.45} and \eqref{eq:2.46},
we get according to the triangle inequality
\begin{align*}
&\chi_{\intt}(\xi)\left|\widehat{\psi}(t,\xi)-\left(\widehat{\ml{J}}_0(t,|\xi|)+\widehat{\ml{N}}_0(t,|\xi|)  \right) \widehat{\Psi}_{1,2}(\xi)-\ml{J}_1(t,|\xi|)\widehat{\Psi}_{0,2}(\xi)\right|\\
&\qquad\lesssim\chi_{\intt}(\xi)\left|\widehat{\psi}(t,\xi)-\widehat{\ml{G}}_1(t,\xi) -\widehat{\ml{G}}_2(t,\xi) \right|+\chi_{\intt}(\xi)\left|\widehat{\ml{G}}_1(t,\xi)-\widehat{\ml{J}}_0(t,|\xi|)\widehat{\Psi}_{1,2}(\xi)-\widehat{\ml{N}}_{0}(t,|\xi|)\widehat{\Psi}_{1,2}(\xi) \right| \\
&\qquad\quad+\chi_{\intt}(\xi)\left|\widehat{\ml{G}}_2(t,\xi)-\widehat{\ml{J}}_1(t,|\xi|)\widehat{\Psi}_{0,2}(\xi) \right| \\
&\qquad\lesssim\chi_{\intt}(\xi)|\xi|\mathrm{e}^{-c|\xi|^2t}\left(|\widehat{\psi}_0(\xi)|+|\widehat{\psi}_1(\xi)|+|\widehat{\psi}_2(\xi)|\right),
\end{align*}
which means by the Hausdorff-Young inequality that
\begin{align*}
		& \left\|\chi_{\intt}(D) \left( \psi(t,\cdot)-\big(\ml{J}_0(t,|D|)+\ml{N}_0(t,|D|)  \big) \Psi_{1,2}(\cdot)-\ml{J}_1(t,|D|)\Psi_{0,2}(\cdot) \right)
	\right\|_{\dot{H}^{k}} \\
	&\qquad \lesssim \left\|\chi_{\intt}(\xi)|\xi|^{k+1}\mathrm{e}^{-c|\xi|^2t}\right\|_{L^2}\|(\psi_0 ,\psi_1 ,\psi_2 )\|_{(L^1 )^3} \\
	&\qquad \lesssim (1+t)^{-\frac{n}{4}-\frac{k+1}{2}}\|(\psi_0 ,\psi_1 ,\psi_2 )\|_{(L^1 )^3}.
\end{align*}
The other cases $\ell=1,2$ can be got by the same way. It completes the proof immediately.
\end{proof}

\begin{remark}\label{Remark_Linear_Profi}
	As investigated in Theorems \ref{Thm_First_Order_Prof} and \ref{Thm_Second_Order_Prof}, we discover some new connections between the MGT equation \eqref{Eq_MGT} and some linearized Kuznetsov's equations. It is one of the novelties for this work. First of all, we denote the solution to the linearized Kuznetsov's equation \eqref{Lienar_Kuznetsov}$_1$ with initial datum $\varphi(0,x)=\varphi_0(x)$ and  $\varphi_t(0,x)=\varphi_1(x)$ by $\varphi(t,x;\varphi_0,\varphi_1)$. According to the derived results in \cite{Ikehata-Todorova-Yordanov=2013,Ikehata=2014,Ikehata-Onoderta=2017}, the next large-time profile for Kuznetsov's equation holds:
	\begin{align*}
		\varphi\big(t,x;\varphi_0,\varphi_1\big)\sim \ml{J}_1(t,|D|)\varphi_0(x)+\ml{J}_0(t,|D|)\varphi_1(x).
	\end{align*}
	The derived results of this work, i.e. Theorems \ref{Thm_First_Order_Prof} and \ref{Thm_Second_Order_Prof}, show the large-time profiles for the viscous MGT equation \eqref{Eq_MGT} as
	\begin{align*}
		\psi(t,x)\sim \ml{J}_0(t,|D|)\big(\psi_1(x)+\tau\psi_2(x)\big),
	\end{align*}
	as well as the further profile satisfying
	\begin{align*}
		&\psi(t,x)-\ml{J}_0(t,|D|)\big(\psi_1(x)+\tau\psi_2(x)\big)\\
		&\qquad\qquad\quad\ \ \sim\ml{J}_1(t,|D|)\big(\psi_0(x)-\tau^2\psi_2(x)\big)+\frac{\delta(\delta-4\tau)}{8}t\Delta\ml{J}_1(t,|D|)\big(\psi_1(x)+\tau\psi_2(x)\big).
	\end{align*}
	By employing the same philosophy as the one in Remark \ref{Rem_MGT_Kuz}, we claim
	\begin{align*}
		\psi(t,x)\sim\varphi\big(t,x;0,\psi_1+\tau\psi_2\big),
	\end{align*}
	for the first-order profile, and
	\begin{align*}
		\psi(t,x)\sim \varphi\big(t,x;\psi_0-\tau^2\psi_2,\psi_1+\tau\psi_2\big)+\frac{\delta(\delta-4\tau)}{8}t\Delta\varphi\big(t,x;0,\psi_1+\tau\psi_2\big),
	\end{align*}
	for the second-order profile as $t\gg1$. The last approximations hint the connection between viscous MGT equation and the linear Kuznetsov's equation for large-time.
\end{remark}
Let us end this subsection with the asymptotic profiles of the solutions to \eqref{Eq_MGT} into the one of $K_2(t,x)\ast \psi_2(x)$, 
which plays essential character to investigate the asymptotic profiles for the nonlinear problem. Comparing \eqref{Representation_Fourier_u} with \eqref{Represent_01}, 
we achieve 
\begin{align*}
	\widehat{K}_2(t,|\xi|) =\frac{-\mathrm{e}^{\lambda_1t}+\cos(\mu_{\mathrm{I}}t)\mathrm{e}^{\mu_{\mathrm{R}}t}}{2\mu_{\mathrm{R}}\lambda_1-\mu_{\mathrm{I}}^2-\mu_{\mathrm{R}}^2-\lambda_1^2}-\frac{(\mu_{\mathrm{R}}-\lambda_1)\sin(\mu_{\mathrm{I}}t)\mathrm{e}^{\mu_{\mathrm{R}}t}}{\mu_{\mathrm{I}}(2\mu_{\mathrm{R}}\lambda_1-\mu_{\mathrm{I}}^2-\mu_{\mathrm{R}}^2-\lambda_1^2)}.
\end{align*}
We recall the observations \eqref{eq:2.81}-\eqref{eq:2.83} with 
\begin{align*}
\partial_{t}^{3} \widehat{\ml{G}}_1(t,\xi) &
\sim -\widehat{\ml{J}}_1(t,|\xi|)|\xi|^{2}\widehat{\Psi}_{1,2}(\xi),\\
\partial_{t}^{3} \widehat{\ml{G}}_1(t,\xi) + \widehat{\ml{J}}_1(t,|\xi|)|\xi|^{2} \widehat{\Psi}_{1,2}(\xi) &\sim\delta\left(\frac{4\tau-\delta}{8}|\xi|^2t+\frac{3}{2}\right)|\xi|^4\widehat{\ml{J}}_0(t,|\xi|)\widehat{\Psi}_{1,2}(\xi)=:\widehat{\ml{N}}_3(t,|\xi|)\widehat{\Psi}_{1,2}(\xi),	
\end{align*}
and
\begin{align*}
	 \partial_{t}^{3} \widehat{\ml{G}}_2(t,\xi) &\sim \widehat{\ml{J}}_0(t,|\xi|)|\xi|^{4}\widehat{\Psi}_{0,2}(\xi).
\end{align*}
In conclusion, since $\widehat{\Psi}_{1,2}(\xi)=\widehat{\psi}_1(\xi)+\tau\widehat{\psi}_2(\xi)$, we say
\begin{align*}
	\partial_{t}^{\ell}\widehat{K}_2(t,|\xi|) \sim \frac{\tau}{|\xi|} \mathrm{e}^{-\frac{\delta}{2}|\xi|^{2} t} \partial_{t}^{\ell} \sin(|\xi| t) 
\end{align*}
as the first-order asymptotic profiles, and 
\begin{align*}
	\partial_{t}^{\ell}\widehat{K}_2(t,|\xi|) - \frac{\tau}{|\xi|} \mathrm{e}^{-\frac{\delta}{2}|\xi|^{2} t} \partial_{t}^{\ell} \sin(|\xi| t) 
\sim \tau\widehat{\ml{N}}_{\ell}(t,|\xi|)
\end{align*}
as  the second-order asymptotic profiles,
where $\ell=0,\dots,3$.
As a result, we immediately have the following estimates.
\begin{coro}\label{coro:2.2}
	Let $\psi_2\in H^{s+1}\cap L^1$ for $s\in[0,\infty)$.
Then, the following decay estimates hold with $\ell=0,\dots,3$:
\begin{align}\label{eq:43}
	 &\left\|(\partial_{t}^{\ell} K_2)(t,\cdot) \ast \psi_2(\cdot)-\left(\frac{\tau}{|D|}\mathrm{e}^{\frac{\delta}{2}\Delta t}\partial_t^{\ell}\sin(|D|t)+\tau\ml{N}_{\ell}(t,|D|)^{m}\right)\psi_2(\cdot) \right\|_{\dot{H}^{s} }\notag\\
&\qquad\lesssim(1+t)^{-\frac{s+\ell+m}{2}-\frac{n}{4}} \|\psi_2 \|_{\dot{H}^{(s-2+\ell)^+} \cap L^1} 	
\end{align}
with $m=0,1$ denoting the consideration of second-order profiles.
\end{coro} 

\subsection{Approximation of solutions to the MGT equation}
Eventually, our main task is to derive the approximation formulas for $\partial_t^{\ell}\psi(t,\cdot)$ with $\ell=0,1,2$, the solutions to \eqref{Eq_MGT}, by diffusion-waves $J_0=J_0(t,x)$ and $J_1=J_1(t,x)$ such that
\begin{align*}
	J_0(t,x):=\mathcal{F}^{-1}_{\xi\to x} \left( \frac{\sin(|\xi|t)}{|\xi|}\mathrm{e}^{-\frac{\delta}{2}|\xi|^2t} \right)
	\ \ \mbox{and}\ \  J_1(t,x):=\mathcal{F}^{-1}_{\xi\to x} \left(\cos(|\xi|t)\,\mathrm{e}^{-\frac{\delta}{2}|\xi|^{2} t} \right),
\end{align*}
as is well-known, which will contribute to the lower bound estimates of the norms of $\psi(t,\cdot)$ as $t\gg1$. 
To do so, strongly motivated by asymptotic profiles, we define the auxiliary functions by the linear combination of diffusion-waves: 
\begin{align*}
\psi^{(1, \ell)}(t,x) := 
\begin{cases} 
(M_{1}+\tau M_{2} )J_0&\mbox{if} \ \ \ell=0,\\ 
(M_{1}+\tau M_{2} )J_1 &\mbox{if} \ \ \ell=1,\\ 
(M_{1}+\tau M_{2} ) \Delta J_0 &\mbox{if} \ \ \ell=2,  
\end{cases}
\end{align*} 
and the further ingenious approximations are 
\begin{align*}
	\psi^{(2, 0)}(t,x) &:= -(M_{1}+\tau M_{2} ) t \frac{\delta(4 \tau-\delta)}{8} \Delta J_1  +(P_{1}+\tau P_{2} ) \circ \nabla J_0 +(M_{0} -\tau^{2} M_{2} )J_1,\\
	\psi^{(2, 1)}(t,x) &:= (M_{1}+\tau M_{2} ) \delta\left(\frac{4\tau-\delta}{8}t\Delta-\frac{1}{2}\right) \Delta J_0  +(P_{1}+\tau P_{2} ) \circ \nabla J_1 - (M_{0} -\tau^{2} M_{2} )\Delta J_0,\\
	\psi^{(2, 2)}(t,x) & := (M_{1}+\tau M_{2} ) \delta\left(\frac{4\tau-\delta}{8}t\Delta-1\right)\Delta J_1  - (P_{1}+\tau P_{2} ) \circ \nabla\Delta J_0 - (M_{0} -\tau^{2} M_{2} )\Delta J_1.
\end{align*}
\begin{prop}\label{prop:2.6}
Under the same assumption as those in Theorem \ref{Thm_First_Order_Prof}, 
the following first-order approximations hold as $t\gg1$:
\begin{align}\label{eq:2.24}
	\left\|\partial_{t}^{\ell} \psi(t,\cdot)-\psi^{(1, \ell)}(t,\cdot)
	\right\|_{\dot{H}^{k}}=o\big(\ml{D}_{n,k+\ell}(t) \big)
\end{align}
for $0 \leqslant k \leqslant s+2-\ell$ with $\ell=0,1,2$.
\end{prop}
\begin{proof}
We only show the estimate \eqref{eq:2.24} with $\ell=0$ in detail.
The other estimates are able to be proved by the same way.
First of all,  we recall the effective facts that 
\begin{align}
 \left\| \ml{J}_j(t,|D|)g(\cdot) - \int_{\mb{R}^n}g(x)\mathrm{d}x\, J_{j}(t,\cdot) \right\|_{\dot{H}^{k}} =o\big(\ml{D}_{n,k+j}(t)\big) \label{eq:2.261}
\end{align}
for all $k \geqslant 0$ and $j=0,1$, by applying the method in \cite[Proposition 1.3]{Ikehata-Takeda-2019} (one may see also \cite{Michihisa-2021}). Now, we decompose the integrand $\psi(t,x)-\psi^{(1, 0)}(t,x)$ into two parts:
\begin{align*}
	\psi(t,x)-\psi^{(1, 0)}(t,x)  = \big( \psi(t,x)-\ml{J}_0(t,|D|) \Psi_{1,2}(x) \big) + \big( \ml{J}_0(t,|D|) \Psi_{1,2}(x)-(M_{1} +\tau M_{2}) J_{0}(t,x) \big).
\end{align*}
Therefore, the estimates \eqref{eq:2.9} and \eqref{eq:2.261} shows
\begin{align*}
	 \left\| \psi(t, \cdot)-\psi^{(1, 0)}(t,\cdot) \right\|_{\dot{H}^{k}} &\lesssim \left\| 
	\psi(t,\cdot)-\ml{J}_0(t,|D|) \Psi_{1,2}(\cdot) 
	\right\|_{\dot{H}^{k}} 
	+ \left\|  \ml{J}_0(t,|D|) \Psi_{1,2}(\cdot)-(M_{1} +\tau M_{2}) J_{0}(t,\cdot) \right\|_{\dot{H}^{k}} \\
	& \lesssim  (1+t)^{-\frac{k}{2}-\frac{n}{4}} \|(\psi_0 ,\psi_1 ,\psi_2 )\|_{\ml{A}_{(k-2)^+}} +o \big(\ml{D}_{n,k}(t)\big),
\end{align*}
as $t\gg1$, which is the desired estimate. We now complete the proof.
\end{proof}

\begin{prop}\label{prop:2.7}
Under the same assumption as those in Theorem \ref{Thm_First_Order_Prof} with additionally $(\psi_{1}, \psi_{2}) \in L^{1}_{1}\times L^{1}_{1}$, the following second-order approximations hold as $t\gg1$: 
\begin{align}\label{eq:2.27}
\left\|\partial_{t}^{\ell} \psi(t,\cdot)-\psi^{(1, \ell)}(t,\cdot)-\psi^{(2, \ell)}(t,\cdot)
\right\|_{\dot{H}^{k}}=o\big(\ml{D}_{n,k+1+\ell}(t) \big)
\end{align}
for $0 \leqslant k \leqslant s+2-\ell$ with $\ell=0,1,2$.
\end{prop}
\begin{proof}
Our strategy for the proof of Proposition \ref{prop:2.7} is similar to the proof of Proposition \ref{prop:2.6}.
We only show the estimate \eqref{eq:2.27} when $\ell=0$.
For this purpose, we divide LHS of \eqref{eq:2.27}  into four components
\begin{align*}
\psi(t,x)-\psi^{(1, 0)}(t,x)-\psi^{(2, 0)}(t,x)  = D_{5}(t,x)+D_{6}(t,x) +D_{7}(t,x) +D_{8}(t,x),	
\end{align*}
where each component is chose by
\begin{align*}
	D_{5}(t,x) & := \psi(t,x)-\left(\ml{J}_0(t,|D|)-t \frac{\delta(4 \tau-\delta)}{8} \Delta \ml{J}_1(t,|D|)  \right) \Psi_{1,2}(x)-\ml{J}_1(t,|D|)\Psi_{0,2}(x), \\
	D_{6}(t,x) & := \ml{J}_0(t,|D|) \Psi_{1,2}(x)-(M_{1} +\tau M_{2}) J_{0}(t,x)-(P_{1} +\tau P_{2}) \circ \nabla J_{0}(t,x), \\
	D_{7}(t,x) & := -t \frac{\delta(4 \tau-\delta)}{8} \Delta \big( \ml{J}_1(t,|D|)  \Psi_{1,2}(x) -(M_{1} +\tau M_{2}) J_{1}(t,x) \big), \\ 
	D_{8}(t,x) & := \ml{J}_1(t,|D|) \Psi_{0,2}(x)-(M_{0} -\tau^2 M_{2}) J_{1}(t,x).
\end{align*}
So, it is easy to see that 
\begin{align*}
	\| D_{5}(t,\cdot) \|_{\dot{H}^{k}}+ \| D_{7}(t,\cdot) \|_{\dot{H}^{k}} + \| D_{8}(t,\cdot) \|_{\dot{H}^{k}} \lesssim  t^{-\frac{k+1}{2}-\frac{n}{4}} \|(\psi_0 ,\psi_1 ,\psi_2 )\|_{\ml{A}_{(k-2)^+}} +o \big(\ml{D}_{n,k+1}(t)\big) 
\end{align*}
as $t \gg 1$ by considering \eqref{eq:2.12} and \eqref{eq:2.261}.
The remainder part of the proof is to verify the estimate 
\begin{align}
	\label{eq:2.292}
	\| D_{6}(t,\cdot) \|_{\dot{H}^{k}}= o \big(\ml{D}_{n,k+1}(t)\big) 
\end{align}
as $t \gg 1$. Recalling the definition of the symbol $\ml{J}_0(t,|D|)$, we observe $D_{6}(t,x)$ may decompose as
\begin{align*}
	D_{6,1}(t,x)& := \int_{|y| \leqslant t^{\frac{1}{8}} } \big(J_{0}(t,x-y) -J_{0}(t,x)-y \circ \nabla J_{0}(t,x)\big) \Psi_{1,2}(y) \mathrm{d}y,\\
	D_{6,2}(t,x)& := \int_{|y| \geqslant t^{\frac{1}{8}} } \big(J_{0}(t,x-y) -J_{0}(t,x) \big) \Psi_{1,2}(y) \mathrm{d}y- \int_{|y| \geqslant t^{\frac{1}{8}} } y \circ \nabla J_{0}(t,x) \Psi_{1,2}(y) \mathrm{d}y.
\end{align*}
For the control of $D_{6,1}(t,x)$, noting that 
\begin{align*}
	|J_{0}(t,x-y) -J_{0}(t,x)-y \circ \nabla J_{0}(t,x) | \lesssim |y|^{2} |\nabla^{2}J_{0}(t, x-\theta_{0} y)|
\end{align*}
equipping $\theta_{0} \in (0,1)$, we have
\begin{align} \label{eq:2.293}
	\|D_{6,1}(t,\cdot) \|_{\dot{H}^{k}} \lesssim t^{\frac{3}{8}} \| \nabla^{2}J_{0}(t, \cdot) \|_{\dot{H}^{k}} \|(\psi_{1}, \psi_{2})  \|_{(L^{1})^{2}} \lesssim t^{-\frac{k}{2}-\frac{1}{8}-\frac{n}{4} }  \|(\psi_{1}, \psi_{2})  \|_{(L^{1})^{2}}.
\end{align} 
Similarly, from the estimate that 
\begin{align*}
|J_{0}(t,x-y) -J_{0}(t,x)| \lesssim |y|\, |\nabla J_{0}(t, x-\theta_{1} y)|	
\end{align*}
for some $\theta_{1} \in (0,1)$, we arrive at the estimate as $t\gg1$ 
\begin{align*}
	\| D_{6,2}(t,\cdot) \|_{\dot{H}^{k}} &  \lesssim\| \nabla J_{0}(t,\cdot) \|_{\dot{H}^{k}}
	\int_{|y| \geqslant t^{\frac{1}{8}} } |y| \big(|\psi_{1}(y) |+ |\psi_{2} (y)|\big) \mathrm{d}y\\
	&   \lesssim t^{-\frac{k}{2}-\frac{n}{4} }  \int_{|y| \geqslant t^{\frac{1}{8}} } |y| \big(|\psi_{1}(y) |+ |\psi_{2} (y)|\big) \mathrm{d}y.
\end{align*}
This immediately leads to
\begin{align} \label{eq:2.294}
	\| D_{6,2}(t,\cdot) \|_{\dot{H}^{k}} =o( t^{-\frac{k}{2}-\frac{n}{4} })  
\end{align}
as $t \gg 1$ by virtue of
\begin{align*}
	\lim_{t \to \infty} \int_{|y| \geqslant t^{\frac{1}{8}} } |y| \big(|\psi_{1}(y) |+ |\psi_{2} (y)|\big) \mathrm{d}y=0,
\end{align*}
since $(\psi_{1}, \psi_{2}) \in L^{1}_{1}\times L^{1}_{1}$.
Combining \eqref{eq:2.293} and \eqref{eq:2.294}, we conclude \eqref{eq:2.292}, 
which proves the result.
\end{proof}
The estimates for first-order approximations $\psi^{(1, \ell)}(t, \cdot)$ and second-order approximations $\psi^{(2, \ell)}(t, \cdot)$ are vital to state the sharpness of the first-order approximation of the solutions to \eqref{Eq_MGT}.
\begin{prop} \label{prop:2.8}
	Let us take the same assumption as those in Theorem \ref{Thm_First_Order_Prof} and $(\psi_{1}, \psi_{2}) \in L^{1}_{1}\times L^{1}_{1}$.
	\begin{description}
		\item[(i)] Assume that $M_{1} +\tau M_{2} \neq 0$.
		Then, the following estimates hold:
		\begin{align}
			\label{eq:2.30}
			|M_{1} +\tau M_{2} | \ml{D}_{n,k+\ell}(t) \lesssim \|\psi^{(1, \ell)}(t, \cdot) \|_{\dot{H}^{k}} \lesssim \ml{D}_{n,k+\ell}(t) \|(\psi_0 ,\psi_1 ,\psi_2 )\|_{\ml{A}_k}
		\end{align}
		for $\ell=0$ when $k=0$ or $k\geqslant 1$, and $\ell=1,2$ when $k\geqslant 0$.
		\item[(ii)] Assume that $A_{0}:=(M_{1} +\tau M_{2})\frac{\tau(4 \tau-\delta)}{8} \neq 0$ or $A_{1}:= M_{0} -\tau^{2} M_{2} \neq 0$ or $\mathbb{B}:= P_{1} +\tau P_{2} \neq  0 $.
		Then, the following estimates hold:
		\begin{align}
			\label{eq:2.31}
			\widetilde{A} \ml{D}_{n,k+1+\ell}(t) \lesssim \|\psi^{(2, \ell)}(t, \cdot) \|_{\dot{H}^{k}} \lesssim \ml{D}_{n,k+1+\ell}(t) \|(\psi_0 ,\psi_1 ,\psi_2 )\|_{\ml{A}_k}
		\end{align}
		for $\ell=0,1,2$ and $k \geqslant 0$ as $t \gg 1$, where $\widetilde{A}=\widetilde{A}(A_{0},A_{1},\mathbb{B})$ is a positive constant.
	\end{description}
\end{prop}
\begin{proof}
The proof of the estimate \eqref{eq:2.30} is well-established (for example, \cite{Ikehata=2014} and \cite{Ikehata-Onoderta=2017}). 
We only prove the estimate \eqref{eq:2.31} for $\ell=0$, since the other cases can be shown by an analogous method.  
Indeed, the direct calculation leads to the upper bound estimates 
\begin{align*}
	\|\psi^{(2, 0)}(t, \cdot) \|_{\dot{H}^{k}} \lesssim (1+t)^{-\frac{k}{2}-\frac{n}{4}} \|(\psi_0 ,\psi_1 ,\psi_2 )\|_{\ml{A}_k}
\end{align*}
by the facts that \eqref{eq:2.261} as well as
\begin{align*}
	t   \| \Delta J_1(t,\cdot ) \|_{\dot{H}^{k}} + \| \nabla J_0(t,\cdot ) \|_{\dot{H}^{k}} +\| J_1(t,\cdot ) \|_{\dot{H}^{k}}  \lesssim (1+t)^{-\frac{k}{2}-\frac{n}{4}}.
\end{align*}
In what follows, we show the estimate from the below
\begin{align*}
	 \|\psi^{(2, 0)}(t, \cdot) \|_{\dot{H}^{k}}\gtrsim (1+t)^{-\frac{k}{2}-\frac{n}{4}}.
\end{align*}
Here, we initially notice that 
\begin{align*}
\widehat{\psi}^{(2, 0)}(t, \xi)= \mathrm{e}^{-\frac{\delta}{2}  |\xi|^{2} t} \left( 
(-A_{0} t |\xi|^{2} +A_{1}) \cos(|\xi| t) + i \mathbb{B} \circ \frac{\xi}{|\xi|} \sin(|\xi|t) 
\right)
\end{align*}
and then 
\begin{align*}
	|\widehat{\psi}^{(2, 0)}(t, \xi)|^{2}= \mathrm{e}^{-\delta |\xi|^{2} t} \left( 
	(-A_{0} t |\xi|^{2} +A_{1})^{2} \cos^{2}(|\xi| t) + \left( \mathbb{B} \circ \frac{\xi}{|\xi|} \right)^2 \sin^{2}(|\xi|t) 
	\right).
\end{align*}
Therefore,
using the simple rules  $2\cos^{2}(|\xi| t)=1-\cos(2 |\xi| t)$, $2\sin^{2}(|\xi| t)=1+\cos(2 |\xi| t)$, and changing the variable such that $\eta = \sqrt{t} \xi$, one gains
\begin{align} \label{eq:2.320}
\| \psi^{(2, 0)}(t, \cdot) \|_{\dot{H}^{k}}^{2} =\|\,|\xi|^k \widehat{\psi}^{(2, 0)}(t, \xi) \|_{L^{2}}^{2}=\widetilde{I}_{1}(t)+\widetilde{I}_{2}(t)+\widetilde{I}_{3}(t)+\widetilde{I}_{4}(t),
\end{align}
where we denoted
\begin{align*}
\widetilde{I}_{1}(t) & := \frac{
t^{-\frac{n}{2}-k}
}{2}   
\int_{\mathbb{R}^{n}} \mathrm{e}^{-\delta |\eta|^{2}} |\eta|^{2k} (-A_{0}|\eta|^{2} +A_{1})^{2} \mathrm{d} \eta, \\
\widetilde{I}_{2}(t) & := -\frac{
t^{-\frac{n}{2}-k}
}{2}   
\int_{\mathbb{R}^{n}} \mathrm{e}^{-\delta |\eta|^{2}} |\eta|^{2k} (-A_{0}|\eta|^{2} +A_{1})^{2} \cos(2 |\eta| \sqrt{t}) \mathrm{d} \eta, \\
\widetilde{I}_{3}(t) & := \frac{
t^{-\frac{n}{2}-k}
}{2}   
\int_{\mathbb{R}^{n}} \mathrm{e}^{-\delta |\eta|^{2}} |\eta|^{2k} \left(\mathbb{B} \circ \frac{\eta}{|\eta|} \right)^{2} \mathrm{d} \eta, \\
\widetilde{I}_{4}(t) & := \frac{
t^{-\frac{n}{2}-k}
}{2}   
\int_{\mathbb{R}^{n}} \mathrm{e}^{-\delta |\eta|^{2}} |\eta|^{2k}\left(\mathbb{B} \circ \frac{\eta}{|\eta|} \right)^{2} \cos(2 |\eta| \sqrt{t}) \mathrm{d} \eta.
\end{align*}
For this reason, we claim that 
\begin{align}
\widetilde{I}_{1}(t)+\widetilde{I}_{3}(t)& \geqslant C t^{-k-\frac{n}{2}}, \label{eq:2.33} \\
\widetilde{I}_{2}(t)+\widetilde{I}_{4}(t) &=o(t^{-k-\frac{n}{2}}), \label{eq:2.34}
\end{align}
as $t \gg1$, where $C=C(|A_{0}|, |A_{1}|)>0$. The estimate \eqref{eq:2.34} is a direct consequence of the Riemann-Lebesgue theorem (cf. \cite{Folland-1995}).
This idea is firstly proposed in \cite{Ikehata=2014}.
More precisely, 
taking into account the polar coordinate, and denoting the measure of $\mathbb{S}^{n-1}$ by $|\mathbb{S}^{n-1}|$,
we easily see that 
\begin{align*}
2t^{k+\frac{n}{2}}\big(\widetilde{I}_{2}(t)+\widetilde{I}_{4}(t)\big) & = -|\mathbb{S}^{n-1}| \int_{0}^{\infty} \mathrm{e}^{-\delta r^{2}} |r|^{2k+n-1} (-A_{0}r^{2} +A_{1})^{2} \cos(2 r \sqrt{t}) \mathrm{d}r \\
& \quad + \int_{0}^{\infty} \mathrm{e}^{-\delta r^{2}} |r|^{2k+n-1} \cos(2 r \sqrt{t}) \mathrm{d}r \int_{\mathbb{S}^{n-1}} (\mathbb{B} \circ \omega)^{2} \mathrm{d} \omega\\
&= - \frac{|\mathbb{S}^{n-1}|}{2} \int_{\mathbb{R}} \mathrm{e}^{-\delta r^{2}} |r|^{2k+n-1} (-A_{0}r^{2} +A_{1})^{2} \cos(2 r \sqrt{t}) \mathrm{d}r \\
& \quad + \frac{1}{2} \int_{\mathbb{R}} \mathrm{e}^{-\delta r^{2}} |r|^{2k+n-1} \cos(2 r \sqrt{t}) \mathrm{d}r \int_{\mathbb{S}^{n-1}} (\mathbb{B} \circ \omega)^{2} \mathrm{d} \omega \\
& = o(1)
\end{align*}
as $t \gg 1$ by the Riemann-Lebesgue theorem,
where we used the facts that $\mathrm{e}^{-\delta r^{2}} |r|^{2k+n-1} (-A_{0}r^{2} +A_{1})^{2} $ and $\mathrm{e}^{-\delta r^{2}} |r|^{2k+n-1}$ belong to $L^{1}$ and they are even functions.
That leads to \eqref{eq:2.34} exactly. To show the estimate \eqref{eq:2.33}, we recall that if $t^{\frac{n}{2}+k} \widetilde{I}_{1}(t) =0$, then both $A_{0}=0$ and $A_{1}=0$.
In other words, if we assume that $A_{0} \neq 0$ or $A_{1} \neq 0$, then there exists a constant $C=C(|A_{0}|, |A_{1}|)>0$ such that $\widetilde{I}_{1}(t) \gtrsim t^{-\frac{n}{2}-k}$. On the other hand, noting that 
\begin{align*}
\int_{\mathbb{S}^{n-1}} \omega_{j} \omega_{k} \mathrm{d} \omega = 
\begin{cases}
 \frac{1}{n} |\mathbb{S}^{n-1}| & \mbox{if} \ \  j=k, \\
 0& \mbox{if} \ \  j \neq k, 
\end{cases}
\end{align*}
where $\omega = (\omega_{1}, \omega_{2}, \dots, \omega_{n})$, 
we use the polar coordinate again to compute 
\begin{align*}
\widetilde{I}_{3}(t)  & = \frac{
t^{-\frac{n}{2}-k}
}{2}   
\int_{0}^{\infty} \mathrm{e}^{-\delta r^{2}} |r|^{2k+n-1} \mathrm{d}r \int_{\mathbb{S}^{n-1}} \left(\mathbb{B} \circ \omega \right)^{2} \mathrm{d} \omega \\
& =
 \frac{
t^{-\frac{n}{2}-k}
}{2} 
|\mathbb{B}|^{2} |\mathbb{S}^{n-1}|
\int_{0}^{\infty} \mathrm{e}^{-\delta r^{2}} |r|^{2k+n-1} \mathrm{d}r \int_{\mathbb{S}^{n-1}} \mathrm{d} \omega,
\end{align*}
which implies that if $\mathbb{B} \neq 0$, we see that  $\widetilde{I}_{3}(t)  \geqslant C t^{-\frac{n}{2}-k}$ for some $C=C(|\mathbb{B}|)>0$ as $t\gg1$. Summing up, we obtain the estimate \eqref{eq:2.33}.
By \eqref{eq:2.320}, \eqref{eq:2.33} and \eqref{eq:2.34}, we have the estimate \eqref{eq:2.31} for $\ell=0$, which is the desired result.
We complete the proof of Proposition \ref{prop:2.8}.
\end{proof}
As some consequences of Propositions \ref{prop:2.6}-\ref{prop:2.7}, we conclude lower bound estimates for $\psi(t,\cdot)$, the solution to \eqref{Eq_MGT}. For example, the idea of the proof has been widely applied in damped wave equation \cite[Equation (1.9)]{Ikehata=2014}, or the methodology of optimal estimates in this paper.  
\begin{theorem}\label{thm:2.3}
Let us take the same assumption as those in Theorem \ref{Thm_First_Order_Prof} and $(\psi_{1}, \psi_{2}) \in L^{1}_{1}\times L^{1}_{1}$.
\begin{description}
	\item[(i)] Suppose that $M_{1}+\tau M_{2} \neq 0$.
	Then, the following optimal estimates hold as $t\gg1$:
	\begin{align*}
		|M_{1}+\tau M_{2}| \ml{D}_{n,k+\ell}(t) \lesssim \|\partial_t^{\ell}\psi(t,\cdot)
		\|_{\dot{H}^{k}}\lesssim \ml{D}_{n,k+\ell}(t)  \|(\psi_0 ,\psi_1 ,\psi_2 )\|_{\ml{A}_k} 
	\end{align*}
	for $\ell=0$ when $k=0$ or $k\geqslant 1$, and $\ell=1,2$ when $k\geqslant 0$.
	\item[(ii)] Suppose that $A_{0} \neq 0$ or $A_{1} \neq 0$ or $\mathbb{B} \neq  0$.
	Then, the following optimal estimates hold as $t\gg1$:
	\begin{align} \label{eq:2.41}
		\widetilde{A} \ml{D}_{n,k+1+\ell}(t) \lesssim \|\partial_t^{\ell}\psi(t,\cdot)-\psi^{(1, \ell)}(t,\cdot)
		\|_{\dot{H}^{k}}\lesssim \ml{D}_{n,k+1+\ell}(t) \|(\psi_0 ,\psi_1 ,\psi_2 )\|_{\ml{A}_k}
	\end{align}
	for $0 \leqslant k \leqslant s+2-\ell$ with $\ell=0,1,2$, where $\widetilde{A}=\widetilde{A}(A_{0},A_{1},\mathbb{B})$ is a positive constant.
\end{description}
\end{theorem}
\begin{remark}
The assumption on (ii) of Theorem \ref{thm:2.3} implies that
we can obtain the lower bound estimates \eqref{eq:2.41},
even if $P_{1}+\tau P_{2} =0$.
\end{remark}

\section{Global (in time) existence of Sobolev solution to the JMGT equation}\label{Section_GESDS_JMGT}
\subsection{Main result and discussion on global (in time) solution}
To begin with this section, let us state our main result for global (in time) existence of solution for the nonlinear Cauchy problem, i.e. the JMGT equation \eqref{JMGT_Dissipative} for any $n\geqslant 1$. 
\begin{theorem}\label{Thm_GESDS}
	Let 
	$s>(n/2-1)^+$ for all $n\geqslant 1$.
	Then, there exists a constant $\epsilon>0$ such that for all $(\psi_0,\psi_1,\psi_2)\in\ml{A}_s$ with $\|(\psi_0,\psi_1,\psi_2)\|_{\ml{A}_s}\leqslant \epsilon$, there is a uniquely determined Sobolev solution
	\begin{align*}
		\psi\in\ml{C}([0,\infty),H^{s+2})\cap\ml{C}^1([0,\infty),H^{s+1})\cap \ml{C}^2([0,\infty),H^s)
	\end{align*}
	to the JMGT equation \eqref{JMGT_Dissipative}. Furthermore, the following estimates hold:
	\begin{align*}
		\|\psi(t,\cdot)\|_{L^2}&\lesssim\ml{D}_n(t)\|(\psi_0,\psi_1,\psi_2)\|_{\ml{A}_s},\\
		\|\psi(t,\cdot)\|_{\dot{H}^1}&\lesssim (1+t)^{-\frac{n}{4}} \|(\psi_0,\psi_1,\psi_2)\|_{\ml{A}_s},\\
		\|\partial_t^{\ell}\psi(t,\cdot)\|_{L^2}&\lesssim (1+t)^{\frac{1}{2}-\frac{\ell}{2}-\frac{n}{4}}\|(\psi_0,\psi_1,\psi_2)\|_{\ml{A}_s}\ \ \mbox{if}\ \ \ell=1,2,\\
		\|\partial_{t}^{\ell}\psi(t,\cdot)\|_{\dot{H}^{s+2-\ell}}&\lesssim (1+t)^{-\frac{1}{2}-\frac{s}{2}-\frac{n}{4}}\|(\psi_0,\psi_1,\psi_2)\|_{\ml{A}_s}\ \ \mbox{if}\ \ \ell=0,1,2,
	\end{align*}
where the time-dependent coefficients $\ml{D}_n(t)$ were defined in Subsection \ref{subsection:Notation}.
\end{theorem}
\begin{remark}
	The derived estimates of solutions to the JMGT equation coincide with those in Propositions \ref{Prop_Estimate_Solution_itself} and \ref{Prop_Estimate_Time_Derivative} for the MGT equation. That is to say that we observe the phenomenon of no loss of decay to the corresponding linear problem \eqref{Eq_MGT}.  
\end{remark}

\begin{remark}\label{Rem_Other_Model}
	Our approach of the proof for Theorem \ref{Thm_GESDS} also can be applied to demonstrate global (in time) existence of Sobolev solutions for
	\begin{itemize}
		\item the Cauchy problem for the following Westervelt's type JMGT equation:
		\begin{align*}
			\tau\psi_{ttt}+\psi_{tt}-\Delta\psi-(\delta+\tau)\Delta\psi_t=\partial_t\left(\left(1+\frac{B}{2A}\right)(\psi_t)^2\right),
		\end{align*}
		where we neglect local nonlinear effects so that $|\nabla\psi|^2\approx(\psi_t)^2$ in the modeling;
		\item the Cauchy problem for the following acoustic pressure equation:
		\begin{align*}
			\tau p_{ttt}+p_{tt}-\Delta p-(\delta+\tau)\Delta p_t=\partial_{tt}\left(\frac{1}{2\rho_0}\left(1+\frac{B}{2A}\right)p^2\right),
		\end{align*}
		with the medium density (constant) $\rho_0>0$, where $p=p(t,x)\in\mb{R}$ denotes the acoustic pressure;
	\end{itemize}
	without additional difficulty (the absence of the gradient nonlinearity makes the problem easier), where the control of nonlinear terms can follow the one in next parts.
\end{remark}

To end this discussion, let us propose three remarks for comparing the previous result in \cite{Racke-Said-2020}.
\begin{remark}
	In the view of global (in time) solution for the JMGT equation \eqref{JMGT_Dissipative} in the three-dimensional case, \cite[Theorem 1.1]{Racke-Said-2020} required $(\psi_0,\psi_1,\psi_2)\in H^{s+2}\times H^{s+2}\times H^{s+1}$ carrying $s>5/2$. Differently from their result, in Theorem \ref{Thm_GESDS} for $n=3$, we assume $(\psi_0,\psi_1,\psi_2)\in H^{s+2}\times H^{s+1}\times H^{s}$ with additional $L^1$ regularity carrying $s>1/2$. In other words, we require less regularity for $\psi_1$ and $\psi_2$. The assumption for additional $L^1$ regularity is beneficial to the demand of higher-order $H^s$ regularity.  
\end{remark}
\begin{remark}
	Concerning estimates for the global (in time) solution and its derivatives for the JMGT equation \eqref{JMGT_Dissipative} in the three-dimensional case,  \cite[Theorem 1.3]{Racke-Said-2020} assumed $H^s\cap L^1$ regularity for initial data and weighted $L^1$	regularity carrying $\int_{\mb{R}^3}\psi_{tt}(0,x)\mathrm{d}x=0=\int_{\mb{R}^3}\psi_{ttt}(0,x)\mathrm{d}x$. Differently from their result, in Theorem \ref{Thm_GESDS} for $n=3$, we derived estimates for each derivative of solution separately (even for the solution itself) instead of a total energy without any special assumption for initial data.
\end{remark}
\begin{remark}
We proved existence of global (in time) solution for the JMGT equation \eqref{JMGT_Dissipative} for any $n\geqslant 1$, even for the physical dimensions $n=1,2$ that did not be consider in \cite{Racke-Said-2020}.
\end{remark}
\subsection{Philosophy of the proof for global (in time) existence of unique solution}
 For the sake of  clarity,
we denote a time-dependent function by
\begin{align*}
		\widetilde{\ml{D}}_{n,s}(t):=\begin{cases}
			(1+t)^{-\frac{s}{2}+\frac{1}{4}}&\mbox{if} \ \ n=1,\\
			(1+t)^{-\frac{s}{2}-\frac{1}{2}} \ln (\mathrm{e}+t)&\mbox{if}\ \ n=2,\\
			(1+t)^{-\frac{s}{2}-\frac{n}{4}}&\mbox{if} \ \ n\geqslant 3,
		\end{cases}
\end{align*}
which is useful to describe the decay properties of the nonlinear term.
Let us construct a sort of time-weighted evolution spaces for any $T>0$ such that
\begin{align*}
	X_s(T):=\ml{C}([0,T],H^{s+2})\cap\ml{C}^1([0,T],H^{s+1})\cap \ml{C}^2([0,T],H^s),
\end{align*}
carrying the corresponding norms
\begin{align*}
	\|\psi\|_{X_s(T)}:=\sup\limits_{t\in[0,T]}&\left(\big(\ml{D}_n(t)\big)^{-1}\|\psi(t,\cdot)\|_{L^2}+(1+t)^{\frac{n}{4}}\|\,|D| \psi(t,\cdot)\|_{L^2}+\sum\limits_{\ell=1,2}(1+t)^{\frac{\ell-1}{2}+\frac{n}{4}}\|\partial_t^{\ell}\psi(t,\cdot)\|_{L^2} \right.\\
	&\quad\left.+\sum\limits_{\ell=0,1,2}(1+t)^{\frac{1}{2}+\frac{s}{2}+\frac{n}{4}}\|\,|D|^{s+2-\ell}\partial_t^{\ell}\psi(t,\cdot)\|_{L^2}\right)
\end{align*}
for $n=1,2$, and 
\begin{align*}
	\|\psi\|_{X_s(T)}:=\sup\limits_{t\in[0,T]}\left(\sum\limits_{\ell=0,1,2}(1+t)^{\frac{\ell-1}{2}+\frac{n}{4}}\|\partial_t^{\ell}\psi(t,\cdot)\|_{L^2}+\sum\limits_{\ell=0,1,2}(1+t)^{\frac{1}{2}+\frac{s}{2}+\frac{n}{4}}\|\,|D|^{s+2-\ell}\partial_t^{\ell}\psi(t,\cdot)\|_{L^2}\right)
\end{align*}
for $n\geqslant 3$, 
where $s>(n/2-1)^+$ for all $n\geqslant 1$.
\begin{remark}
In the case when $n\geqslant 3$, we may get the bounded estimates for $(1+t)^{\frac{n}{4}}\|\,|D|\psi(t,\cdot)\|_{L^2}$ from the definition of $X_s(T)$ by using some interpolations as follows:
\begin{align*}
	(1+t)^{\frac{n}{4}}\|\psi(t,\cdot)\|_{\dot{H}^1}\lesssim\left((1+t)^{-\frac{1}{2}+\frac{n}{4}}\|\psi(t,\cdot)\|_{L^2}\right)^{\frac{s+1}{s+2}}\left((1+t)^{\frac{1}{2}+\frac{s}{2}+\frac{n}{4}}\|\psi(t,\cdot)\|_{\dot{H}^{s+2}}\right)^{\frac{1}{s+2}}\lesssim\|\psi\|_{X_s(t)},
\end{align*}
 whereas we need to construct weighted $\dot{H}^1$ norm in $X_s(T)$ when $n=1,2$ as a supplement.
\end{remark}
We may introduce the operator $N$ by
\begin{align}\label{*N4}
	N:\ \psi(t,x)\in X_s(T)\to N\psi(t,x):=\psi^{\lin} (t,x)+\psi^{\non}(t,x),
\end{align}
where $\psi^{\lin} =\psi^{\lin} (t,x)$ is the solution to the linearized Cauchy problem \eqref{Eq_MGT}, and the integral operator is denoted by
\begin{align}\label{*N1}
	\psi^{\non}(t,x):=\int_0^tK_2(t-\sigma,x)\ast f\big(\psi(\sigma,x)\big)\mathrm{d}\sigma,
\end{align}
where $K_2(t,x)$ is the kernel for the third data of the linearized problem \eqref{Eq_MGT}, which is motivated by Duhamel's principle. Here, we denote by $f(\psi)$ the nonlinear term on the right-hand side of \eqref{JMGT_Dissipative}$_1$, precisely,
\begin{align*}
	f(\psi)  :=\partial_t\left(\frac{B}{2A}(\psi_t)^2+|\nabla\psi|^2\right)=\frac{B}{A}\psi_t\,\psi_{tt}+2\nabla\psi\cdot\nabla\psi_t, 
\end{align*}
and the quadratic nonlinear term is
\begin{align*}
	f_{0}(\psi):=\frac{B}{2A}(\psi_t)^2+|\nabla\psi|^2\ \ \mbox{so that}\ \ f(\psi)=\partial_tf_0(\psi).
\end{align*}
Indeed, they hold
\begin{align*}
\widehat{K}_{2}(0,|\xi|)=\partial_{t} \widehat{K}_{2}(0,|\xi|)=0,\ \ \partial_{t}^{2} \widehat{K}_{2}(0,|\xi|)=1
\end{align*}
 moreover,
\begin{align} \label{eq:28}
	\partial_{t}^{\ell} \psi^{\non}(t,x)= \int_0^t \partial_{t}^{\ell} K_2(t-\sigma,x)\ast f\big(\psi(\sigma,x)\big)\mathrm{d}\sigma
\end{align}
for $\ell=0,1,2$. Using these notations and integration by parts in regard to $t$,  we may arrive at 
\begin{align}\label{eq:29} 
	\partial_{t}^{\ell} \psi^{\non}(t,x)= -(\partial_{t}^{\ell} K_2)(t,x)\ast f_{0}\big(\psi(0,x)\big)-\int_0^t \partial_{t}^{\ell+1} K_2(t-\sigma,x)\ast f_{0}\big(\psi(\sigma,x)\big)\mathrm{d}\sigma.
\end{align}
for $\ell=0,1$, and
\begin{align}
	\partial_{t}^{2} \psi^{\non}(t,x)=f_{0}\big(\psi(t,x)\big) -(\partial_{t}^{2} K_2)(t,x)\ast f_{0}\big(\psi(0,x)\big)-\int_0^t \partial_{t}^{3} K_2(t-\sigma,x)\ast f_{0}\big(\psi(\sigma,x)\big)\mathrm{d}\sigma. \label{eq:30}
\end{align}
Occasionally, the following decomposition is useful to obtain sharp decay properties in some norms:
\begin{align}\label{eq:33}
		 \partial_{t}^{\ell} \psi^{\non}(t,x)
	& = (\partial_{t}^{\ell} K_2)(\tfrac{t}{2},x)\ast f_{0}\big(\psi(\tfrac{t}{2},x)\big) -(\partial_{t}^{\ell} K_2)(t,x)\ast f_{0}\big(\psi(0,x)\big)\notag\\
	&\quad -\int_0^\frac{t}{2} \partial_{t}^{\ell+1} K_2(t-\sigma,x)\ast f_{0}\big(\psi(\sigma,x)\big)\mathrm{d}\sigma +\int_\frac{t}{2}^{t} \partial_{t}^{\ell} K_2(t-\sigma,x)\ast f\big(\psi(\sigma,x)\big)\mathrm{d}\sigma,
\end{align}
for $\ell=0,1,2$,
which was derived by \eqref{eq:28} and the use of integration by parts in the interval $[0,t/2]$. We should underline that the nonlinear term is $f_0(\psi)$ in the integral on $[0,t/2]$, while the nonlinear term is $f(\psi)=\partial_tf_0(\psi)$ in another integral on $[t/2,t]$. That is one of the strategic part to overcome some technical difficulties in lower dimensional cases.

In the next parts, we will demonstrate global (in time) existence and uniqueness of small data Sobolev solutions to the JMGT equation \eqref{JMGT_Dissipative} via proving a fixed point of operator $N$ which means $N\psi\in X_s(T)$ for any $T>0$. Namely, the following two crucial inequalities:
\begin{align}
	\|N\psi\|_{X_s(T)}&\lesssim\|(\psi_0,\psi_1,\psi_2)\|_{\ml{A}_{s} }+\|(\psi_0,\psi_1,\psi_2)\|_{\ml{A}_{s} }^2+\|\psi\|_{X_s(T)}^2,\label{Cruc-01}\\
	\|N\psi-N\bar{\psi}\|_{X_s(T)}&\lesssim\|\psi-\bar{\psi}\|_{X_s(T)}\left(\|\psi\|_{X_s(T)}+\|\bar{\psi}\|_{X_s(T)}\right),\label{Cruc-02}
\end{align}
will be proved, respectively, for any $T>0$ with the data space $\ml{A}_s$ defined in Subsection \ref{subsection:Notation}. In our second inequality \eqref{Cruc-02}, $\psi$ and $\bar{\psi}$ are two solutions to the viscous JMGT equation \eqref{JMGT_Dissipative}. Providing that we take $\|(\psi_0,\psi_1,\psi_2)\|_{\ml{A}_{s} }=\epsilon$ as a sufficiently small and positive constant, then we combine \eqref{Cruc-01} with \eqref{Cruc-02} to claim that there exists global (in time) small data Sobolev solution $\psi^{*}=\psi^{*}(t,x)\in X_s(T)$ by using Banach's fixed point theorem.

With the aim of  demonstrating the crucial inequalities \eqref{Cruc-01} and \eqref{Cruc-02}, we will divide our proof into two steps (two subsections, respectively) as follows:
\begin{description}
	\item[Step A.]  Under the construction of evolution spaces $X_s(T)$, we control the nonlinear terms in $L^m$ and $\dot{H}^s$ norms with $m=1,2$ and suitable $s>(n/2-1)^+$ by some interpolations and embedding theorems.
	\item[Step B.] With the aid of the derived estimates in Step A, we rigorously prove our desired estimates and complete the proof by suitable representations \eqref{eq:29}-\eqref{eq:33}.
\end{description}

\subsection{Step A. Some estimates for nonlinear terms}\label{Subsec_Est_Nonlinear_Term}
Our purpose in this part is to give some estimates to the nonlinear terms in some norms as follows: $f_0(\psi)$ in $L^1$, $L^2$, $\dot{H}^1$, $\dot{H}^s$ as well as $\dot{H}^{s+1}$ with $n\geqslant 1$; and $f(\psi)$ in $\dot{H}^{s}$  with $n=1$  for any $s>(n/2-1)^+$
which will be controlled by $\|\psi\|_{X_s(T)}^2$ with suitable time-dependent coefficients for any $\sigma\in[0,T]$.
\medskip

\noindent\textbf{Substep A1.} Estimates of $f_{0}(\psi)$ in some norms for $n \geqslant 1$.\\
The estimate for the $L^{1}$ norm is clearly obtained by
\begin{align} \label{Est_02}
	\big\|f_{0}\big(\psi(\sigma,\cdot)\big)\big\|_{L^1}&\lesssim\|\psi_t(\sigma,\cdot)\|_{L^2}^{2}+\|\nabla\psi(\sigma,\cdot)\|_{L^2}^2 \lesssim (1+\sigma)^{-\frac{n}{2}}\|\psi\|_{X_s(\sigma)}^2.
\end{align}
Next, we show the estimate for the $L^{2}$ norm. Applying H\"older's inequality and the Sobolev inequality \eqref{eq:C1} equipping $s+1>n/2$, we arrive at
\begin{align}\label{Est_03}
	\big\|f_{0}\big(\psi(\sigma,\cdot)\big)\big\|_{L^2}&\lesssim\|\psi_t(\sigma,\cdot) \|_{L^{\infty}} \| \psi_{t}(\sigma,\cdot)\|_{L^2}+\|\nabla\psi(\sigma,\cdot) \|_{L^{\infty}} \|\nabla\psi(\sigma,\cdot)\|_{L^2}\notag\\
&\lesssim\|\psi_t(\sigma,\cdot) \|_{L^{2}}^{2-\frac{n}{2(s+1)}} \| \psi_{t}(\sigma,\cdot)\|_{\dot{H}^{s+1}}^{\frac{n}{2(s+1)}}+\|\nabla \psi(\sigma,\cdot) \|_{L^{2}}^{2-\frac{n}{2(s+1)}} \| \nabla \psi (\sigma,\cdot)\|_{\dot{H}^{s+1}}^{\frac{n}{2(s+1)}}\notag \\
& \lesssim (1+\sigma)^{-\frac{3n}{4}}\|\psi\|_{X_s(\sigma)}^2.
\end{align}
For the estimate in the $\dot{H}^{s+1}$ norm, we apply the fractional Leibniz rule (see Proposition \ref{fractionleibnizrule}) and the Sobolev inequality \eqref{eq:C1} to have 
\begin{align} \label{eq:3.1}
	\big\|f_{0}\big(\psi(\sigma,\cdot)\big)\big\|_{\dot{H}^{s+1}}&\lesssim\|\psi_t(\sigma,\cdot) \|_{L^{\infty}} \|\, |D|^{s+1} \psi_{t}(\sigma,\cdot)\|_{L^2}+\|\nabla\psi(\sigma,\cdot) \|_{L^{\infty}} \|\,|D|^{s+1}\nabla \psi(\sigma,\cdot)\|_{L^2}\notag\\
	& \lesssim (1+\sigma)^{-\frac{3n}{4}-\frac{s+1}{2}}\|\psi\|_{X_s(\sigma)}^2,
\end{align}
in which the definition of the evolution space $X_s(T)$ was considered. Once we have \eqref{Est_03} and \eqref{eq:3.1}, they ensure
\begin{align}
\big\| f_{0}\big(\psi(\sigma,\cdot)\big)\big\|_{\dot{H}^{1}}& \lesssim (1+\sigma)^{-\frac{3n}{4}-\frac{1}{2}}\|\psi\|_{X_s(\sigma)}^2,  \label{eq:35} \\
\big\| f_{0}\big(\psi(\sigma,\cdot)\big)\big\|_{\dot{H}^{s}}& \lesssim (1+\sigma)^{-\frac{3n}{4}-\frac{s}{2}}\|\psi\|_{X_s(\sigma)}^2,\notag
\end{align}
by suitable interpolations between $L^2$ and $\dot{H}^{s+1}$.

\medskip
\noindent\textbf{Substep A2.} Estimate of $f(\psi)$ in the $\dot{H}^s$ norms for $n=1$.\\
Being differ from Substep A1, we employ the fractional Leibniz rule (see Proposition \ref{fractionleibnizrule}) to deduce
\begin{align*}
	\big\|f\big(\psi(\sigma,\cdot)\big)\big\|_{\dot{H}^s}&\lesssim\|\psi_t(\sigma,\cdot)\psi_{tt}(\sigma,\cdot)\|_{\dot{H}^s}+\|\nabla\psi(\sigma,\cdot)\cdot\nabla\psi_t(\sigma,\cdot)\|_{\dot{H}^s}\\
	&\lesssim \|\,|D|^s\psi_t(\sigma,\cdot)\|_{L^{\infty}}\|\psi_{tt}(\sigma,\cdot)\|_{L^{2}}+\|\psi_t(\sigma,\cdot)\|_{L^{\infty}}\|\psi_{tt}(\sigma,\cdot)\|_{\dot{H}^s}\\
	&\quad+\|\,|D|^{s+1}\psi(\sigma,\cdot)\|_{L^{\infty}}\|\,|D|\psi_t(\sigma,\cdot)\|_{L^{2}}+\|\,|D|\psi(\sigma,\cdot)\|_{L^{\infty}}\|\,|D|\psi_t(\sigma,\cdot)\|_{\dot{H}^s}\\
	&=:I_1(\sigma)+I_2(\sigma)+I_3(\sigma)+I_4(\sigma).
\end{align*}
By applying the Sobolev inequality \eqref{eq:C2}  associated with the Gagliardo-Nirenberg inequality in Proposition \ref{fractionalgagliardonirenbergineq}, one notices
\begin{align*}
I_1(\sigma)&\lesssim\|\psi_t(\sigma,\cdot)\|_{\dot{H}^s}^{\frac{1}{2}} \|\psi_t(\sigma,\cdot)\|_{\dot{H}^{s+1}}^{\frac{1}{2}} \|\psi_{tt}(\sigma,\cdot)\|_{L^2}\\
&\lesssim\|\psi_t(\sigma,\cdot)\|_{L^2}^{\frac{1}{2(s+1)}}\|\psi_t(\sigma,\cdot)\|_{\dot{H}^{s+1}}^{\frac{2s+1}{2(s+1)}}\|\psi_{tt}(\sigma,\cdot)\|_{L^2}\\
&\lesssim (1+\sigma)^{-\frac{5}{4}-\frac{s}{2}}\|\psi\|_{X_s(T)}^2.
\end{align*} 
Repeating the same way, we conclude
\begin{align*}
	I_2(\sigma)&\lesssim \|\psi_t(\sigma,\cdot)\|_{L^{2}}^{\frac{1}{2}} \|\psi_t(\sigma,\cdot)\|_{\dot{H}^{1}}^{\frac{1}{2}} \|\psi_{tt}(\sigma,\cdot)\|_{\dot{H}^s}\\
	&\lesssim\|\psi_t(\sigma,\cdot)\|_{L^{2}}^{\frac{2s+1}{2(s+1)}}\|\psi_t(\sigma,\cdot)\|_{\dot{H}^{s+1}}^{\frac{1}{2(s+1)}}\|\psi_{tt}(\sigma,\cdot)\|_{\dot{H}^s},
\end{align*}
and
\begin{align*}
I_3(\sigma)&\lesssim\|\psi(\sigma,\cdot)\|_{\dot{H}^{s+1}}^{\frac{1}{2}}\|\psi(\sigma,\cdot)\|_{\dot{H}^{s+2}}^{\frac{1}{2}}\|\psi_t(\sigma,\cdot)\|_{\dot{H}^{1}}\\
&\lesssim\|\psi(\sigma,\cdot)\|_{\dot{H}^1}^{\frac{1}{2(s+1)}}\|\psi(\sigma,\cdot)\|_{\dot{H}^{s+2}}^{\frac{2s+1}{2(s+1)}}\|\psi_t(\sigma,\cdot)\|_{L^2}^{\frac{s}{s+1}}\|\psi_t(\sigma,\cdot)\|_{\dot{H}^{s+1}}^{\frac{1}{s+1}},
\end{align*}
where we are benefit from the $\|\,|D|\psi(t,\cdot)\|_{L^2}$ norm in the evolution space, as well as
\begin{align*}
	I_4(\sigma)&\lesssim \|\psi(\sigma,\cdot)\|_{\dot{H}^{1}}^{\frac{1}{2}} \|\psi(\sigma,\cdot)\|_{\dot{H}^{2}}^{\frac{1}{2}}\|  \psi_t(\sigma,\cdot)\|_{\dot{H}^{s+1}}\\
	&\lesssim\|\psi(\sigma,\cdot)\|_{\dot{H}^1}^{\frac{2s+1}{2(s+1)}}\|\psi(\sigma,\cdot)\|_{\dot{H}^{s+2}}^{
	\frac{1}{2(s+1)}}\|  \psi_t(\sigma,\cdot)\|_{\dot{H}^{s+1}}.
\end{align*}
As a consequence, we may find
\begin{align}
	\big\|f\big(\psi(\sigma,\cdot)\big)\big\|_{\dot{H}^s} \lesssim I_1(\sigma)+I_2(\sigma) +I_3(\sigma)+I_4(\sigma) \lesssim (1+\sigma)^{-\frac{5}{4}-\frac{s}{2}}\|\psi\|_{X_s(T)}^2 \label{eq:3.3}
\end{align}
for $n=1$, since $\|\cdot\|_{X_s(\sigma)}\leqslant \|\cdot\|_{X_s(T)}$ for any $0\leqslant \sigma\leqslant T$.

\subsection{Step B. Proof of desired estimates in the evolution space} \label{subs:3.4}
First of all, according to the derived estimates in Propositions 2.1 as well as 2.2, we claim that
\begin{align*}
	\|\psi^{\lin} \|_{X_s(T)}\lesssim\|(\psi_0,\psi_1,\psi_2)\|_{\ml{A}_s}.
\end{align*}
In other words, concerning the aim \eqref{Cruc-01}, we just need to estimate $\|\psi^{\non}\|_{X_s(T)}$ in the remaindering of this subsection.
We firstly show the estimates for $\partial_{t}^{\ell} \psi^{\non}(t,\cdot)$ carrying $\ell=0,1$.
By using the expression \eqref{eq:29}, and then dividing the interval into $[0,t/2]$ as well as $[t/2,t]$, we may see
\begin{align}\label{*N2}
	 \|\partial_{t}^{\ell} \psi^{\non}(t,\cdot)\|_{L^2} & \lesssim \ml{D}_{n,\ell}(t) \|(\psi_0,\psi_1,\psi_2)\|_{\ml{A}_s}^2+\int_0^{\frac{t}{2}} (1+t-\sigma)^{-\frac{\ell}{2}-\frac{n}{4}} \big\|f_{0} \big(\psi(\sigma,\cdot)\big)\big\|_{L^2\cap L^1}\mathrm{d}\sigma  \notag\\
	& \quad
	+\int_{\frac{t}{2}}^{t} (1+t-\sigma)^{-\frac{\ell}{2}}\big\|f_{0}\big(\psi(\sigma,\cdot)\big)\big\|_{L^2}\mathrm{d}\sigma \notag\\
	&\lesssim \ml{D}_{n,\ell}(t) \|(\psi_0,\psi_1,\psi_2)\|_{\ml{A}_s}^2+  (1+t)^{-\frac{\ell}{2}-\frac{n}{4}} \int_{0}^{\frac{t}{2}} (1+\sigma)^{-\frac{n}{2}} \mathrm{d}\sigma\,\|\psi\|_{X_s(T)}^2 \notag\\
     & \quad
	+  (1+t)^{-\frac{3n}{4}} \int_{\frac{t}{2}}^{t} (1+t-\sigma)^{-\frac{\ell}{2}} \mathrm{d}\sigma\,\|\psi\|_{X_s(T)}^2 \notag\\
	&\lesssim \ml{D}_{n,\ell}(t) \|(\psi_0,\psi_1,\psi_2)\|_{\ml{A}_s}^2 + \widetilde{\ml{D}}_{n,\ell}(t) \|\psi\|_{X_s(T)}^2
\end{align}
by derived $(L^2\cap L^1)-L^2$ estimates \eqref{Est_04} for $[0,t/2]$ and $L^2-L^2$ estimates \eqref{Est_05} for $[t/2,t]$.  
In the above, we used the fact that
\begin{align*}
	\big\|(\partial_t^{\ell}K_2)(t,\cdot)\ast f_0\big(\psi(0,\cdot)\big)\big\|_{L^2}
	\lesssim\ml{D}_{n,\ell}(t)\left(\|\psi_1\|_{L^4\cap L^2}^2+\|\nabla\psi_0\|_{L^4\cap L^2}^2\right)\lesssim\ml{D}_{n,\ell}(t)\|(\psi_0,\psi_1,\psi_2)\|_{\ml{A}_s}^2
\end{align*}
for $\ell=0,1$.  Similarly, in the case  $\ell=2$ with $n\geqslant 2$, we have 
\begin{align*}
	 \|\partial_{t}^{2} \psi^{\non}(t,\cdot)\|_{L^2}&\lesssim \big\|f_0\big(\psi(t,\cdot)\big)\big\|_{L^2}+(1+t)^{-\frac{1}{2}-\frac{n}{4}}\big\|f_0\big(\psi(0,\cdot)\big)\big\|_{L^2\cap L^1}\\
	 &\quad+\int_0^{\frac{t}{2}}(1+t-\sigma)^{-1-\frac{n}{4}}\big\|f_0\big(\psi(\sigma,\cdot)\big)\big\|_{\dot{H}^1\cap L^1}\mathrm{d}\sigma+\int_{\frac{t}{2}}^t(1+t-\sigma)^{-\frac{1}{2}}\big\|f_0\big(\psi(\sigma,\cdot)\big)\big\|_{\dot{H}^1}\mathrm{d}\sigma\\
	 & \lesssim \ml{D}_{n,2}(t) \|(\psi_0,\psi_1,\psi_2)\|_{\ml{A}_s}^2 + \widetilde{\ml{D}}_{n,2}(t) \|\psi\|_{X_s(T)}^2.
\end{align*} 
 For the remainder case $\ell=2$ with $n=1$, 
noting that 
\begin{align*}
\big\| (\partial_{t}^{2} K_2)(\tfrac{t}{2},\cdot)\ast f_{0}\big(\psi(\tfrac{t}{2},\cdot)\big) \big\|_{L^{2}} 
&\lesssim (1+t)^{-\frac{3}{4}} \big\| f_{0}\big(\psi(\tfrac{t}{2},\cdot)\big)  \big\|_{L^2\cap L^{1}}\\
&\lesssim (1+t)^{-\frac{5}{4}} \|\psi\|_{X_s(\sigma)}^2
\end{align*}
by employing \eqref{Est_04}, \eqref{Est_02} and \eqref{Est_03}, 
we may apply the derived estimates \eqref{eq:2.3} with $s=0$, $\ell=2$ for $[0,t/2]$ and \eqref{Est_05} to \eqref{eq:33} to see that 
\begin{align*}
 \| \partial_{t}^{2} \psi^{\non}(t,\cdot) \|_{L^{2}}
& \lesssim \ml{D}_{1,2}(t) \|(\psi_0,\psi_1,\psi_2)\|_{\ml{A}_s}^2 + \widetilde{\ml{D}}_{1,2}(t) \|\psi\|_{X_s(T)}^2 \\
& \quad+\int_0^\frac{t}{2} (1+t-\sigma)^{-\frac{5}{4}} \big\| f_{0}\big(\psi(\sigma,\cdot)\big) \big\|_{\dot{H}^{1} \cap L^{1}}\mathrm{d}\sigma
+\int_\frac{t}{2}^{t}  (1+t-\sigma)^{-\frac{1}{2}}  \big\| f\big(\psi(\sigma,\cdot)\big) \big\|_{L^2} \mathrm{d}\sigma \\
& \lesssim \ml{D}_{1,2}(t) \|(\psi_0,\psi_1,\psi_2)\|_{\ml{A}_s}^2 + \widetilde{\ml{D}}_{1,2}(t) \|\psi\|_{X_s(T)}^2+ (1+t)^{-\frac{5}{4}} \int_0^\frac{t}{2} (1+\sigma)^{-\frac{1}{2}} \mathrm{d}\sigma\,\|\psi\|_{X_s(T)}^2 \\
& \quad +  (1+t)^{-\frac{5}{4}} \int_\frac{t}{2}^{t}  (1+t-\sigma)^{-\frac{1}{2}}   \mathrm{d}\sigma\,\|\psi\|_{X_s(T)}^2 \\
& \lesssim \ml{D}_{1,2}(t) \|(\psi_0,\psi_1,\psi_2)\|_{\ml{A}_s}^2 + \widetilde{\ml{D}}_{1,2}(t) \|\psi\|_{X_s(T)}^2
\end{align*}
by concerning estimates \eqref{Est_02}, \eqref{eq:35} and \eqref{eq:3.3} with $s=0$.  Consequently,
\begin{align*}
	\big(\ml{D}_{n,\ell}(t)\big)^{-1}\| \partial_{t}^{\ell} \psi^{\non}(t,\cdot)\|_{L^{2}}\lesssim\|(\psi_0,\psi_1,\psi_2)\|_{\ml{A}_s}^2 + \|\psi\|_{X_s(T)}^2 
\end{align*}
for any $\ell=0,1,2$ because of $\widetilde{\ml{D}}_{n,\ell}(t)\lesssim\ml{D}_{n,\ell}(t)$.
%

Now, we derive the estimate for $|D| \psi^{\non}(t,\cdot)$ in the $L^2$ norm for $n=1,2$.
We apply estimates \eqref{eq:16} and \eqref{Est_09} equipping $s=0$, $\ell=1$ to \eqref{eq:29} to obtain 
\begin{align*}
	\|\,|D| \psi^{\non}(t,\cdot)\|_{L^2} 
	&\lesssim \ml{D}_{n,1}(t) \|(\psi_0,\psi_1,\psi_2)\|_{\ml{A}_s}^2+  (1+t)^{-\frac{1}{2}-\frac{n}{4}} \int_{0}^{\frac{t}{2}} (1+\sigma)^{-\frac{n}{2}} \mathrm{d}\sigma\, \|\psi\|_{X_s(T)}^2 \\
& \quad
+  (1+t)^{-\frac{3n}{4}} \int_{\frac{t}{2}}^{t}(1+t-\sigma)^{-\frac{1}{2}} \mathrm{d}\sigma\,\|\psi\|_{X_s(T)}^2 \\
	&\lesssim \ml{D}_{n,1}(t) \|(\psi_0,\psi_1,\psi_2)\|_{\ml{A}_s}  + \widetilde{\ml{D}}_{n,1}(t)  \|\psi\|_{X_s(T)}^2. 
\end{align*}

Finally, concerning higher-order derivative of solution, by the formula \eqref{eq:29}, 
we will use \eqref{eq:2.3}, \eqref{Est_02} and \eqref{eq:3.1} in the sub-interval $[0,t/2]$ and \eqref{eq:2.4}, \eqref{eq:3.1} in the sub-interval $[t/2,t]$.
Precisely, since 
\begin{align*} 
\big\||D|^{s+2-\ell}(\partial_t^{\ell}K_2)(t,\cdot)\ast f_0\big(\psi(0,\cdot)\big)\big\|_{L^2}&\lesssim (1+t)^{-\frac{1}{2}-\frac{s}{2}-\frac{n}{4}}\left(\|\,|\psi_1|^2\|_{\dot{H}^s\cap L^1}+\|\,|\nabla\psi_0|^2\|_{\dot{H}^s\cap L^1}\right)\\
&\lesssim (1+t)^{-\frac{1}{2}-\frac{s}{2}-\frac{n}{4}}\|(\psi_0,\psi_1,\psi_2)\|_{\ml{A}_s}^2,
\end{align*}
 when $\ell=0,1$, one may notice
\begin{align}\label{*N3}
	 \|\,|D|^{s+2-\ell}\partial_t^{\ell}\psi^{\non}(t,\cdot)\|_{L^2}&\lesssim \ml{D}_{n,s+2}(t) \|(\psi_0,\psi_1,\psi_2)\|_{\ml{A}_s}^2+\int_0^{\frac{t}{2}} (1+t-\sigma)^{-1-\frac{s}{2}-\frac{n}{4}} \big\|f_{0} \big(\psi(\sigma,\cdot)\big)\big\|_{\dot{H}^{s+1}\cap L^1}\mathrm{d}\sigma   \notag\\
	& 
	\quad+\int_{\frac{t}{2}}^{t}(1+t-\sigma)^{-\frac{1}{2}} \big\|f_{0}\big(\psi(\sigma,\cdot)\big)\big\|_{\dot{H}^{s+1}}\mathrm{d}\sigma \notag \\
	&\lesssim \ml{D}_{n,s+2}(t) \|(\psi_0,\psi_1,\psi_2)\|_{\ml{A}_s}^2 +  (1+t)^{-1-\frac{s}{2}-\frac{n}{4}} \int_{0}^{\frac{t}{2}} (1+\sigma)^{-\frac{n}{2}} \mathrm{d}\sigma\,\|\psi\|_{X_s(T)}^2\notag\\
& 
	\quad+  (1+t)^{-\frac{3n}{4}-\frac{s+1}{2}} \int_{\frac{t}{2}}^{t} (1+t-\sigma)^{-\frac{1}{2}} \mathrm{d}\sigma\,\|\psi\|_{X_s(T)}^2 \notag\\
	&\lesssim \ml{D}_{n,s+2}(t) \|(\psi_0,\psi_1,\psi_2)\|_{\ml{A}_s}^2  + \widetilde{\ml{D}}_{n,s+2}(t)  \|\psi\|_{X_s(T)}^2. 
\end{align} 
In the case $\ell=2$ with $n\geqslant 2$, we apply the same argument to the last one, which implies 
\begin{align*}
	 \|\,|D|^{s}\partial_t^{2}\psi^{\non}(t,\cdot)\|_{L^2}&\lesssim \ml{D}_{n,s+2}(t) \|(\psi_0,\psi_1,\psi_2)\|_{\ml{A}_s}^2+(1+t)^{-\frac{3n}{4}-\frac{s}{2}}\|\psi\|_{X_s(T)}^2  \\
	&\quad +\int_0^{\frac{t}{2}} (1+t-\sigma)^{-1-\frac{s}{2}-\frac{n}{4}} \big\|f_{0} \big(\psi(\sigma,\cdot)\big)\big\|_{\dot{H}^{s+1}\cap L^1}\mathrm{d}\sigma\\ 
	&\quad +\int_{\frac{t}{2}}^{t}(1+t-\sigma)^{-\frac{1}{2}} \big\|f_{0}\big(\psi(\sigma,\cdot)\big)\big\|_{\dot{H}^{s+1}}\mathrm{d}\sigma \\
	&\lesssim \ml{D}_{n,s+2}(t) \|(\psi_0,\psi_1,\psi_2)\|_{\ml{A}_s}^2  + \widetilde{\ml{D}}_{n,s+2}(t)  \|\psi\|_{X_s(T)}^2.
\end{align*}
 On the other hand, 
when $\ell=2$ with $n=1$, we consider the representation \eqref{eq:33} and obtained estimates \eqref{eq:2.3}, \eqref{Est_02} and \eqref{eq:3.1} for $[0,t/2]$ and \eqref{Est_05}, \eqref{eq:3.3} for $[t/2, t]$ to deduce
\begin{align*}
 &\|\,|D|^{s} \partial_{t}^{2} \psi^{\non}(t,\cdot) \|_{L^{2}}\\
  &\qquad \lesssim \ml{D}_{1,s+2}(t) \|(\psi_0,\psi_1,\psi_2)\|_{\ml{A}_s}^2 + \widetilde{\ml{D}}_{1,s+2}(t) \|\psi\|_{X_s(T)}^2 \\
&\qquad\quad +\int_0^\frac{t}{2} (1+t-\sigma)^{-\frac{s}{2}-\frac{5}{4}} \big\| f_{0}\big(\psi(\sigma,\cdot)\big) \big\|_{\dot{H}^{s+1} \cap L^{1}}\mathrm{d}\sigma
+\int_\frac{t}{2}^{t}  (1+t-\sigma)^{-\frac{1}{2}}  \big\| f\big(\psi(\sigma,\cdot)\big) \big\|_{\dot{H}^{s}} \mathrm{d}\sigma \\
&\qquad \lesssim \ml{D}_{1,s+2}(t) \|(\psi_0,\psi_1,\psi_2)\|_{\ml{A}_s}^2 + \widetilde{\ml{D}}_{1,s+2}(t) \|\psi\|_{X_s(T)}^2 \\
&\qquad\quad +\|\psi\|_{X_s(T)}^2 (1+t)^{-\frac{s}{2}-\frac{5}{4}} \int_0^\frac{t}{2} (1+\sigma)^{-\frac{1}{2}} \mathrm{d}\sigma
+\|\psi\|_{X_s(T)}^2  (1+\sigma)^{-\frac{5}{4}-\frac{s}{2}} \int_\frac{t}{2}^{t}  (1+t-\sigma)^{-\frac{1}{2}}   \mathrm{d}\sigma \\
&\qquad \lesssim \ml{D}_{1,s+2}(t) \|(\psi_0,\psi_1,\psi_2)\|_{\ml{A}_s}^2 + \widetilde{\ml{D}}_{1,s+2}(t) \|\psi\|_{X_s(T)}^2.
\end{align*}

%
%
All in all, summarizing the above all derived estimates and making use of $\widetilde{\ml{D}}_{n,s+2}(t)\lesssim\ml{D}_{n,s+2}(t)$ for any $n\geqslant 1$, we directly conclude
\begin{align*}
	\|\psi^{\non}\|_{X_s(T)}\lesssim\|(\psi_0,\psi_1,\psi_2)\|_{\ml{A}_s}^2+\|\psi\|_{X_s(T)}^2,
\end{align*}
which says that the desired estimate \eqref{Cruc-01} holds if $s>(n/2-1)^+$ for all $n\geqslant 1$.

Let us sketch the proof of the Lipschitz condition via remarking
\begin{align*}
	\big\|\,|D|^{s+2-\ell}\partial_t^{\ell}\big(N\psi(t,\cdot)-N\bar{\psi}(t,\cdot)\big)\big\|_{L^2}=\left\|\,|D|^{s+2-\ell}\partial_t^{\ell}\int_0^tK_2(t-\sigma,\cdot)\ast \left(f\big(\psi(\sigma,\cdot)\big)-f\big(\bar{\psi}(\sigma,\cdot)\big)\right)\mathrm{d}\sigma\right\|_{L^2}
\end{align*}
for $s\in\{\ell-2\}\cup(0,\infty)$ and $\ell=0,1,2$. By using some estimates in Propositions \ref{Prop_Estimate_Solution_itself} and \ref{Prop_Estimate_Time_Derivative} again, we just need to estimate
$f_{0}(\psi)-f_{0}(\bar{\psi})$ in $L^1$, $L^2$ as well as $\dot{H}^{s+1}$ norms for $n \geqslant 1$; and $f_{0}(\psi)-f_{0}(\bar{\psi})$ in $\dot{H}^s$ norm for $n=1$.
Actually, we may estimate the difference of the nonlinear terms by
\begin{align*}
	|f_{0}(\psi)-f_{0}(\bar{\psi})|&\lesssim|(\psi_t)^{2}-(\bar{\psi}_t)^{2}|+|\, |\nabla\psi|^{2} -|\nabla\bar{\psi}|^{2}|\\
	&\lesssim|\psi_t(\psi_{t}-\bar{\psi}_{t})|+|(\psi_t-\bar{\psi}_t)\bar{\psi}_{t}|+|\nabla\psi\cdot\nabla(\psi-\bar{\psi})|+|\nabla(\psi-\bar{\psi})\cdot\nabla\bar{\psi}|
\end{align*}
and
\begin{align*}
	|f(\psi)-f(\bar{\psi})|&\lesssim|\psi_t\psi_{tt}-\bar{\psi}_t\bar{\psi}_{tt}|+|\nabla\psi\cdot\nabla\psi_t-\nabla\bar{\psi}\cdot\nabla\bar{\psi}_t|\\
	&\lesssim|\psi_t(\psi_{tt}-\bar{\psi}_{tt})|+|(\psi_t-\bar{\psi}_t)\bar{\psi}_{tt}|+|\nabla\psi\cdot\nabla(\psi_t-\bar{\psi}_t)|+|\nabla(\psi-\bar{\psi})\cdot\nabla\bar{\psi}_t|.
\end{align*}
Therefore, by repeating the same procedure as Subsection \ref{Subsec_Est_Nonlinear_Term}, we obtain
\begin{align*}
	\big\|f_{0}\big(\psi(\sigma,\cdot)\big)-f_{0}\big(\bar{\psi}(\sigma,\cdot)\big)\big\|_{L^m}\lesssim\|\psi-\bar{\psi}\|_{X_s(T)}\left(\|\psi\|_{X_s(T)}+\|\bar{\psi}\|_{X_s(T)}\right) (1+\sigma)^{-\frac{nm}{4}-\frac{n}{4}},
\end{align*}
where $m=1,2$. 
Furthermore, the next estimates hold:
\begin{align*}
	\big\|f_{0}\big(\psi(\sigma,\cdot)\big)-f_{0}\big(\bar{\psi}(\sigma,\cdot)\big)\big\|_{\dot{H}^{s+1}}\lesssim\|\psi-\bar{\psi}\|_{X_s(T)}\left(\|\psi\|_{X_s(T)}+\|\bar{\psi}\|_{X_s(T)}\right)
		(1+\sigma)^{-\frac{s+1}{2}-\frac{3n}{4}}
\end{align*}
for $n \geqslant 1$, and 
\begin{align*}
	\big\|f\big(\psi(\sigma,\cdot)\big)-f\big(\bar{\psi}(\sigma,\cdot)\big)\big\|_{\dot{H}^s}\lesssim\|\psi-\bar{\psi}\|_{X_s(T)}\left(\|\psi\|_{X_s(T)}+\|\bar{\psi}\|_{X_s(T)}\right)
		(1+\sigma)^{-\frac{s}{2}-\frac{5}{4}}
\end{align*}
for $n=1$.
Then, following the same steps as those in the proof of \eqref{Cruc-01}, we can conclude \eqref{Cruc-02}. The proof is complete.

\section{Large-time profiles for the JMGT equation}\label{Sec-4}
\subsection{First-order approximation by diffusion-waves}
Before stating Theorem \ref{thm:4.1}, we describe the first-order profiles of the nonlinear portion of solution to \eqref{JMGT_Dissipative} in the sense of the integral form (Duhamel's principle in \eqref{*N1}) by using $\ml{J}_0(t,|D|)$ and $\ml{J}_1(t,|D|)$, which were defined in \eqref{Symbol_J}.
\begin{prop}\label{N_Prop_First_Order_Prof}
Let $n \geqslant 2$.
The nonlinear part of the global (in time) solution constructed in Theorem \ref{Thm_GESDS} satisfies the following refined estimates:
\begin{align*}
	\left\|\partial_t^{\ell}\psi^{\non}(t,\cdot) + \frac{\tau}{|D|}\mathrm{e}^{\frac{\delta}{2}\Delta t}\partial_t^{\ell}\sin(|D|t) f_{0}\big(\psi(0,\cdot)\big)
	\right\|_{\dot{H}^{k}}&\lesssim (1+t)^{-\frac{k+\ell}{2}-\frac{n}{4}}\|(\psi_0,\psi_1,\psi_2)\|_{\ml{A}_s}
\end{align*}
for $0\leqslant k\leqslant s+2-\ell$ with $\ell=0,1,2$. 
\end{prop}
\begin{proof}
From Theorem \ref{Thm_GESDS}, it is clear that $\|\psi\|_{X_s(T)}\lesssim\|(\psi_0,\psi_1,\psi_2)\|_{\ml{A}_s}$, which is small. We recall the derived estimates in Subsection \ref{subs:3.4}, e.g. \eqref{*N2} and \eqref{*N3} to have 
\begin{align}\label{eq:68}
\left\| 
\int_0^t \partial_{t}^{\ell+1} K_2(t-\sigma,\cdot)\ast f_{0}\big(\psi(\sigma,\cdot)\big)\mathrm{d}\sigma
\right\|_{L^2}
\lesssim \widetilde{\ml{D}}_{n,\ell}(t)	\|(\psi_0,\psi_1,\psi_2)\|_{\ml{A}_s},
\end{align}
and 
\begin{align}\label{eq:69}
\left\| |D|^{s+2-\ell}
\int_0^t \partial_{t}^{\ell+1} K_2(t-\sigma,\cdot)\ast f_{0}\big(\psi(\sigma,\cdot)\big)\mathrm{d}\sigma
\right\|_{L^2}
\lesssim \widetilde{\ml{D}}_{n,s+2}(t)\|(\psi_0,\psi_1,\psi_2)\|_{\ml{A}_s},
\end{align}
where we took $\ell=0,1,2$. Therefore, when $\ell=0$ in our aim estimate, we may use the representation \eqref{eq:29} associated with \eqref{eq:68}, \eqref{eq:69} and \eqref{eq:43} with $m=0$ as well as the replacement of source term to see
\begin{align*} 
	& \big\|\psi^{\non}(t,\cdot) + \tau \ml{J}_0(t,|D|) f_{0}\big(\psi(0, \cdot)\big)
	\big\|_{\dot{H}^{k}} \\
	&\qquad\lesssim \big\| K_{2} (t,\cdot) \ast f_{0}\big(\psi(0,\cdot)\big) - \tau \ml{J}_0(t,|D|) f_{0}\big(\psi(0, \cdot)\big)
	\big\|_{\dot{H}^{k}}
	+ \widetilde{\ml{D}}_{n,k}(t) \|(\psi_0,\psi_1,\psi_2)\|_{\ml{A}_s}\\
	&\qquad\lesssim
	 (1+t)^{-\frac{n}{4}-\frac{k}{2}}\|(\psi_0,\psi_1,\psi_2)\|_{\ml{A}_s}.
\end{align*}
The other cases can be shown in a same way, so we omit the detail.
\end{proof}
Once we have Proposition \ref{N_Prop_First_Order_Prof}, by combining with \eqref{eq:2.261}, 
we may easily obtain the approximation formulas of the nonlinear portion with the aid of diffusion-waves, in which we employed the triangle inequality.
\begin{coro}\label{N_coro_First_Order_Prof}
Let $n \geqslant 2$.
Then, the global (in time) small data solution constructed in Theorem \ref{Thm_GESDS} fulfills the following estimates:
\begin{align*}
	\left\|\partial_{t}^{\ell} \psi^{\non}(t,\cdot)-\widetilde{\psi}^{(1, \ell)}(t,\cdot)
	\right\|_{\dot{H}^{k}}=o\big(\ml{D}_{n,k+\ell}(t) \big)
\end{align*}
as $t \gg 1$ for $0 \leqslant k \leqslant s+2-\ell$ with $\ell=0,1,2$,
where we define
\begin{align*}
\widetilde{\psi}^{(1, \ell)}(t,x) := 
\begin{cases} 
-\tau M_{00}J_0&\mbox{if} \ \ \ell=0,\\ 
- \tau M_{00} J_1 &\mbox{if} \ \ \ell=1,\\ 
-\tau M_{00} \Delta J_0 &\mbox{if} \ \ \ell=2,
\end{cases}
\end{align*} 
and 
\begin{align*}
M_{00} := \int_{\mathbb{R}^{n}} \left(\frac{B}{2A}|\psi_{1}(x)|^2+|\nabla\psi_{0}(x)|^2 \right)\mathrm{d}x.
\end{align*}
\end{coro}
Moreover, summing up Propositions \ref{prop:2.6} and \ref{N_Prop_First_Order_Prof}, as well as Corollary \ref{N_coro_First_Order_Prof} by applying \eqref{*N4}, 
we conclude the first-order approximation of the global (in time) solution to \eqref{JMGT_Dissipative} and the lower bounds of norms.
\begin{theorem}\label{thm:4.1}
Let us take the same assumption as those in Theorem \ref{Thm_GESDS} with $n \geqslant 2$.
\begin{description}
	\item[(i)] Then, the refined estimates hold
	\begin{align*}
	\left\|\partial_{t}^{\ell} \psi(t,\cdot)-\psi^{(1, \ell)}(t,\cdot) -\widetilde{\psi}^{(1, \ell)}(t,\cdot)
	\right\|_{\dot{H}^{k}}=o\big(\ml{D}_{n,k+\ell}(t) \big)
\end{align*}
as $t \gg 1$ for $0 \leqslant k \leqslant s+2-\ell$ with $\ell=0,1,2$.
	\item[(ii)] Suppose that $M_{1}+\tau M_{2}-\tau M_{00} \neq 0$. Then, the following estimates hold:
	\begin{align*}
		|M_{1}+\tau M_{2}-\tau M_{00}| \ml{D}_{n,k+\ell}(t) \lesssim \|\partial_t^{\ell}\psi(t,\cdot)
		\|_{\dot{H}^{k}}\lesssim \ml{D}_{n,k+\ell}(t)  \|(\psi_0 ,\psi_1 ,\psi_2 )\|_{\ml{A}_s} 
	\end{align*}
	as $t \gg 1$ for $\ell=0$ when $k=0$ or $k\geqslant 1$, and $\ell=1,2$ when $k\geqslant0$.
\end{description}
\end{theorem}
\begin{remark}
It is easily seen that for one-dimensional case, the estimates in Corollary \ref{N_coro_First_Order_Prof} and Theorem \ref{thm:4.1} are valid when $\ell=0$ and $k=0$,
since $\ml{D}_{1,k}(t)=\widetilde{\ml{D}}_{1,k}(t)$ for all $k \geqslant 1$.
\end{remark}
\begin{remark}\label{Rem-4.n}
	According to Theorem \ref{thm:4.1}, the first-order profile of global (in time) small data solution to the JMGT equation \eqref{JMGT_Dissipative} can be described by diffusion-waves such that
	\begin{align*}
	\psi(t,x)\sim\int_{\mb{R}^n}\left(\psi_1(x)+\tau\psi_2(x)-\frac{\tau B}{2A}|\psi_1(x)|^2-\tau|\nabla\psi_0(x)|^2\right)\mathrm{d}x\, J_0(t,x)
	\end{align*}
as $t\gg1$. Moreover, the second conclusion of Theorem \ref{thm:4.1} tells us the optimality of derived estimates in Theorem \ref{Thm_GESDS}.
\end{remark}

\subsection{Second-order approximation by diffusion-waves}
This subsection is devoted to the study of second-order profiles and second-order approximation by diffusion-waves of the global (in time) small data solutions to the nonlinear problem \eqref{JMGT_Dissipative}.
To do so, one may study the  crucial estimates for the integral term in the expression formulas in \eqref{eq:28}-\eqref{eq:30}.
\begin{prop} \label{prop:4.2}
Let $n \geqslant 3$. Then, the following refined estimates for the kernel hold for $t \gg 1$:
\begin{align}\label{eq:72}
\left\| 
\int_0^t \partial_{t} K_2(t-\sigma,\cdot)\ast f_{0}\big(\psi(\sigma,\cdot)\big)\mathrm{d}\sigma - M_{\non} J_{1}(t,\cdot)
\right\|_{\dot{H}^{k}}
&=o(t^{-\frac{k}{2}-\frac{n}{4}})\ \ \ \ \mbox{for}\ \ 0\leqslant k\leqslant s+2, \\
\left\| 
\int_0^t \partial_{t}^{2} K_2(t-\sigma,\cdot)\ast f_{0}\big(\psi(\sigma,\cdot)\big)\mathrm{d}\sigma - M_{\non} \Delta J_{0}(t,\cdot)
\right\|_{\dot{H}^{k}}
&=o(t^{-\frac{k+1}{2}-\frac{n}{4}})\ \ \mbox{for}\ \ 0\leqslant k\leqslant s+1,\label{eq:73} \\
\left\| 
\int_0^t \partial_{t}^{3} K_2(t-\sigma,\cdot)\ast f_{0}\big(\psi(\sigma,\cdot)\big)\mathrm{d}\sigma - M_{\non} \Delta J_{1}(t,\cdot)
\right\|_{\dot{H}^{k}}
&=o(t^{-\frac{k+2}{2}-\frac{n}{4}})\ \  \mbox{for}\ \ 0\leqslant k\leqslant s, \label{eq:74}
\end{align}
where we denote
\begin{align*}
M_{\non} := \int_{0}^{\infty} \int_{\mathbb{R}^{n}} f_{0}\big(\psi(\sigma,x)\big)  \mathrm{d}x\mathrm{d} \sigma.
\end{align*}
\end{prop}
\begin{remark}
Our assumption $n \geqslant 3$ is to ensure $|M_{\non}| < \infty$, which comes from the obtained estimate \eqref{Est_02} and $(1+\sigma)^{-\frac{n}{2}} \in L^{1}(0,\infty)$ holding for any $n \geqslant 3$.
\end{remark}
\begin{proof}
The proofs of estimates \eqref{eq:72}-\eqref{eq:74} are quite similar. We remark that due to smallness assumption on initial datum, one notices $\|\psi\|_{X_s(T)}\lesssim 1$.
Here, we only prove the estimate \eqref{eq:72} with $k=0$.
At first, we decompose the integrand into five parts
\begin{align*}
\int_0^t \partial_{t} K_2(t-\sigma,x)\ast f_{0}\big(\psi(\sigma,x)\big)\mathrm{d}\sigma - M_{\non} J_{1}(t,x)
=\sum\limits_{k=5}^{9}I_{k}(t,x),
\end{align*}
where
\begin{align*}
I_{5}(t,x) & := \int_0^{t^{\frac{1}{2}}}
\partial_{t} K_2(t-\sigma,x) \ast f_{0}\big(\psi(\sigma,x)\big)\mathrm{d}\sigma -\int_0^{t^{\frac{1}{2}}} \ml{J}_1(t-\sigma,|D|) f_{0}\big(\psi(\sigma,x)\big)\mathrm{d}\sigma,\\
I_{6}(t,x) & := \int_0^{t^{\frac{1}{2}}} \big( \ml{J}_1(t-\sigma,|D|) - \ml{J}_1(t,|D|)\big)  f_{0}\big(\psi(\sigma,x)\big)\mathrm{d}\sigma,\\
I_{7}(t,x) & := \int_0^{t^{\frac{1}{2}}} \int_{\mathbb{R}^{n} } \big( J_1(t,x-y) - J_1(t,x) \big)  f_{0}\big(\psi(\sigma,y)\big) \mathrm{d}y\mathrm{d}\sigma,\\
I_{8}(t,x) & := -J_1(t,x) \int_{t^{\frac{1}{2}}}^{\infty} \int_{\mathbb{R}^{n} } f_{0}\big(\psi(\sigma,y)\big) \mathrm{d}y \mathrm{d}\sigma,\\
I_9(t,x)&:=\int_{t^{\frac{1}{2}}}^t\partial_tK_2(t-\sigma,x)\ast f_0\big(\psi(\sigma,x)\big)\mathrm{d}\sigma,
\end{align*}
as $t\gg 1$ for $n\geqslant 3$.
The first term in the above is estimated as follows: 
\begin{align*}
\| I_{5}(t,\cdot) \|_{L^{2}} & \lesssim \int_0^{t^{\frac{1}{2}}} (1+t-\sigma)^{-\frac{1}{2}-\frac{n}{4}} \big\| f_{0}\big(\psi(\sigma,\cdot)\big) \big\|_{L^{1} \cap L^{2} } \mathrm{d}\sigma \lesssim  (1+t)^{-\frac{1}{2}-\frac{n}{4}}\|(\psi_0,\psi_1,\psi_2)\|_{\ml{A}_s}^2, 	
\end{align*}
where we used the estimates \eqref{eq:43} with $\ell=1$ and $m=0$, \eqref{Est_02} and \eqref{Est_03}.
For the second component $I_{6}(t,x)$, applying the mean value theorem with respect to $t$, 
we see that  
\begin{align*}
|J_{1}(t-\sigma,x-y)-J_{1}(t,x-y)| \lesssim |\sigma \partial_{t} J_{1}(t-\theta \sigma, x-y)|  
\end{align*}
for $\theta \in (0,1)$. 
Therefore we arrive at the estimate: 
\begin{align*}
\| I_{6}(t,\cdot) \|_{L^{2}} & \lesssim \int_0^{t^{\frac{1}{2}}} \sigma \big\| \partial_{t} J_{1}(t-\theta \sigma, \cdot) \ast f_{0}\big(\psi(\sigma,\cdot)\big) \big\|_{L^{2} } \mathrm{d}\sigma \notag\\
& \lesssim  (1+t)^{-\frac{1}{2}-\frac{n}{4}} \int_0^{t^{\frac{1}{2}}} \sigma \big\| f_{0}\big(\psi(\sigma,\cdot)\big) \big\|_{L^{1} } \mathrm{d}\sigma +\mathrm{e}^{-ct} t^{-1} \int_0^{t^{\frac{1}{2}}} \sigma \| f_{0}(\psi(\sigma,\cdot)) \|_{L^{2} } \mathrm{d}\sigma \notag\\
& \lesssim  (1+t)^{-\frac{1}{2}-\frac{n}{4}} \int_0^{t^{\frac{1}{2}}} (1+\sigma)^{-\frac{n}{2}+1} \mathrm{d}\sigma +\mathrm{e}^{-ct} t^{-1} \int_0^{t^{\frac{1}{2}}} (1+\sigma)^{-\frac{3n}{4}} \mathrm{d}\sigma \notag\\
& \lesssim  t^{-\frac{1}{2}-\frac{n}{4}} \times 
\begin{cases}
	 t^{\frac{1}{4}}& \text{if}\ \ n=3, \\ 
	 \ln (\mathrm{e}+t)& \text{if}\ \ n=4, \\  
	 1& \text{if}\ \ n \geqslant 5,  
\end{cases}
\end{align*}
where we used the estimate 
\begin{align*}
\| \partial_{t} J_{1}(t, \cdot) \ast \psi_{2}(\cdot) \|_{L^{2}} \lesssim (1+t)^{-\frac{1}{2}-\frac{n}{4}} \| \psi_{2} \|_{L^{1}} +\mathrm{e}^{-ct} t^{-1} \| \psi_{2} \|_{L^{2}}.
\end{align*}
To show the estimate for $I_{7}(t,x)$, we split it into two parts as follows: 
\begin{align*}
I_{7}(t,x) & := I_{7,1}(t,x) +I_{7,2}(t,x),
\end{align*}
where
\begin{align*}
I_{7,1}(t,x) & := \int_0^{
	t^{\frac{1}{2}}
} \int_{|y| \leqslant t^{\frac{1}{4} } } \big( J_1(t,x-y) - J_1(t,x) \big)  f_{0}\big(\psi(\sigma,y)\big) \mathrm{d}y\mathrm{d}\sigma,\\
I_{7,2}(t,x) & := \int_0^{t^{\frac{1}{2}}} \int_{|y| \geqslant t^{\frac{1}{4} } } \big( J_1(t,x-y) - J_1(t,x) \big)  f_{0}\big(\psi(\sigma,y)\big) \mathrm{d}y\mathrm{d}\sigma.
\end{align*}
Now, we observe that 
\begin{align*}
\|\, |D| J_{1}(t, \cdot) \ast \psi_{2}(\cdot) \|_{L^{2}} \lesssim (1+t)^{-\frac{1}{2}-\frac{n}{4}} \| \psi_{2} \|_{L^{1}} +\mathrm{e}^{-ct}  \| \psi_{2} \|_{\dot{H}^{1}}
\end{align*}
by direct calculations.
We also have 
\begin{align*}
|J_{1}(t, x-y)-J_{1}(t,x)| \lesssim |y|\, |\nabla J_{1}(t,x- \tilde{\theta} y)|
\end{align*}
for some $\tilde{\theta} \in (0,1)$, 
by the mean value theorem for spatial variables $x$.
Combining these facts, we are able to investigate 
\begin{align*}
\| I_{7,1}(t,\cdot) \|_{L^{2}} & \lesssim t^{\frac{1}{4}} \| \nabla J_{1}(t,\cdot- \tilde{\theta} y) \|_{L^{2}}  
\int_0^{t^{\frac{1}{2}}} \int_{|y| \leqslant t^{\frac{1}{4} } }  \big| f_{0}\big(\psi(\sigma,y)\big) \big|\mathrm{d}y \mathrm{d}\sigma \\
& \lesssim t^{-\frac{1}{4}-\frac{n}{4}} 
\int_0^{t^{\frac{1}{2}}} (1+\sigma)^{-\frac{n}{2}} \mathrm{d}\sigma \lesssim t^{-\frac{1}{4}-\frac{n}{4}}.
\end{align*}
On the other hand, noting that 
\begin{align*}
\lim_{t \to \infty} \int_{0}^{\infty} \int_{|y| \geqslant t^{\frac{1}{4}}} \big|f_{0}\big(\psi(\sigma, y)\big) \big| \mathrm{d}y \mathrm{d}\sigma =0, 
\end{align*}
we see that 
\begin{align*}
\| I_{7,2}(t,\cdot) \|_{L^{2}} & \lesssim t^{-\frac{n}{4}}  \int_{0}^{\infty} \int_{|y| \geqslant t^{\frac{1}{4}}} \big|f_{0}\big(\psi(\sigma, y)\big) \big| \mathrm{d}y \mathrm{d}\sigma =o(t^{-\frac{n}{4}})
\end{align*}
as $t \to \infty$.
From the estimate
\begin{align*}
\left| \int_{t^{\frac{1}{2}}}^{\infty} \int_{\mathbb{R}^{n} } f_{0}\big(\psi(\sigma,y)\big) \mathrm{d}y\mathrm{d}\sigma \right| 
& \lesssim \int_{t^{\frac{1}{2}}}^{\infty} \big\| f_{0}\big(\psi(\sigma,\cdot)\big) \big\|_{L^{1}} \mathrm{d}\sigma \\
& \lesssim \int_{t^{\frac{1}{2}}}^{\infty} (1+\sigma)^{-\frac{n}{2}} \mathrm{d}\sigma \lesssim (1+t)^{\frac{1}{2}-\frac{n}{4}},
\end{align*}
we can immediately get 
\begin{align*}
\| I_{8}(t,\cdot) \|_{L^{2}} \lesssim (1+t)^{\frac{1}{2}-\frac{n}{2}}=o(t^{-\frac{n}{4}})
\end{align*}
for $n\geqslant 3$. Finally, by separating $[t^{\frac{1}{2}},t]$ into $[t^{\frac{1}{2}},t/2]$ and $[t/2,t]$ due to $t/2\gg t^{\frac{1}{2}}$ as $t\gg1$, and employing $(L^2\cap L^1)-L^2$ as well as $L^2-L^2$ estimates, we obtain
\begin{align*}
		\| I_{9}(t,\cdot) \|_{L^{2}} & \lesssim \int_{t^{\frac{1}{2}}}^{\frac{t}{2}} (1+t-\sigma)^{-\frac{n}{4}} \big\| f_{0}\big(\psi(\sigma,\cdot)\big) \big\|_{L^{1} \cap L^{2}} \mathrm{d}\sigma 
+ \int_{\frac{t}{2}}^{t} \big\| f_{0}\big(\psi(\sigma,\cdot)\big) \big\|_{L^{2}} \mathrm{d}\sigma \\
& \lesssim (1+t)^{-\frac{n}{4}}  \int_{t^{\frac{1}{2}}}^{\frac{t}{2}} (1+\sigma)^{-\frac{n}{2}} \mathrm{d}\sigma 
+ \int_{\frac{t}{2}}^{t} (1+\sigma)^{-\frac{3n}{4}} \mathrm{d}\sigma \\
& \lesssim (1+t)^{\frac{1}{2}-\frac{n}{2}}=o(t^{-\frac{n}{4}}).
\end{align*}
Summing up all derived estimates in the above, we claim \eqref{eq:72} directly, which is the desired result.
\end{proof}
Secondly, as the direct consequence of the estimates \eqref{eq:43} and \eqref{eq:2.261}, 
we have the second-order profiles and the second-order approximation by diffusion waves of $\partial_{t}^{\ell} K_{2}(t, x) \ast f_{0}(\psi(0,x))$ for $\ell=0,1,2$.
\begin{coro}\label{coro:4.2}
Let $n \geqslant 1$ and $s>(n/2-1)^{+}$.
Suppose that $(\psi_{0}, \psi_{1}) \in (H^{s+2}\cap L^1)\times (H^{s+1}\cap L^1)$ and $|\nabla \psi_{0}|, \psi_{1} \in L^{1}_{1}$.
Then, the following refined estimates hold:
\begin{align} \label{eq:81}
	\left\|\partial_{t}^{\ell} K_{2}(t, \cdot) \ast f_{0}\big(\psi(0,\cdot)\big)+\widetilde{\psi}^{(1, \ell)}(t,\cdot)-\phi^{(2, \ell)}(t,\cdot)
	\right\|_{\dot{H}^{k}}=o\big(\ml{D}_{n,k+\ell+1}(t) \big)
\end{align}
as $t \to \infty$ for $0 \leqslant k \leqslant s+2-\ell$ with $\ell=0,1,2$,
where we choose
\begin{align*}
	\phi^{(2, 0)}(t,x) &:= -\tau M_{00}  t \frac{\delta(4 \tau-\delta)}{8} \Delta J_1  +\tau P_{00} \circ \nabla J_0 -\tau^{2} M_{00} J_1,\\
	\phi^{(2, 1)}(t,x) &:= -\tau M_{00} \delta\left(-\frac{4\tau-\delta}{8}t\Delta+\frac{1}{2}\right) \Delta J_0  +\tau P_{00}  \circ \nabla J_1  +\tau^{2} M_{00} \Delta J_0,\\
	\phi^{(2, 2)}(t,x) & := -\tau M_{00} \delta\left(-\frac{4\tau-\delta}{8}t\Delta+1\right)\Delta J_1 -\tau P_{00}  \circ \nabla\Delta J_0 +\tau^{2} M_{00} \Delta J_1,
\end{align*}
carrying
\begin{align*}
P_{00}:= \int_{\mathbb{R}^{n}} x f_{0} \big(\psi(0,x)\big) \mathrm{d}x.
\end{align*}
\end{coro}
\begin{proof}
By the assumption of initial datum, we easily see that 
\begin{align*}
|P_{00}| \lesssim \| \psi_{1} \|_{L^{\infty}} \| x \psi_{1} \|_{L^{1}} +\| \,|D| \psi_{0} \|_{L^{\infty}} \| x |D| \psi_{0} \|_{L^{1}}< \infty.
\end{align*}
Therefore, we can apply the same argument as those in the proof of Proposition \ref{prop:2.7} to have the estimate \eqref{eq:81}.
\end{proof}
We can now formulate the second-order approximation by diffusion-waves of the global (in time) small data solutions to \eqref{JMGT_Dissipative} and the lower bound for the norms of the first-order approximation.
\begin{theorem}\label{thm:4.2}
Let us take the same assumption as those in Theorem \ref{Thm_GESDS} with $n \geqslant 3$.
In addition, assume that $|\nabla \psi_{0}|, \psi_{1} \in L^{1}_{1}$.
\begin{description}
	\item[(i)] Then, the refined estimates hold
	\begin{align*}
	\left\|\partial_{t}^{\ell} \psi(t,\cdot)-\sum_{j=1,2} \big(\psi^{(j, \ell)}(t,\cdot) +\widetilde{\psi}^{(j, \ell)}(t,\cdot)\big)
	\right\|_{\dot{H}^{k}}=o\big(\ml{D}_{n,k+\ell+1}(t) \big)
\end{align*}
as $t\gg 1$ for $0 \leqslant k \leqslant s+2-\ell$ with $\ell=0,1,2$,
where 
\begin{align*}
	\widetilde{\psi}^{(2, 0)}(t,x) &:= -\tau M_{\non} J_{1} -\phi^{(2,0)},\\
	\widetilde{\psi}^{(2, 1)}(t,x) &:= -\tau M_{\non} \Delta J_{0} -\phi^{(2,1)},\\
	\widetilde{\psi}^{(2, 2)}(t,x) & := -\tau M_{\non} \Delta J_{1} -\phi^{(2,2)}.
\end{align*}
	\item[(ii)] Assume that $\widetilde{A}_{0}:=(M_{1} +\tau M_{2}+\tau M_{00})\frac{\tau(4 \tau-\delta)}{8} \neq 0$ or $\widetilde{A}_{1}:= M_{0} -\tau^{2} M_{2}+\tau^{2} M_{00}-\tau M_{\non} \neq 0$ or $\widetilde{\mathbb{B}}:= P_{1} +\tau P_{2}-\tau P_{00} \neq  0$.
		Then, the optimal estimates hold 
		\begin{align*}
			\widetilde{C} \ml{D}_{n,k+1+\ell}(t) \lesssim \| \partial_{t}^{\ell} \psi(t,\cdot)-\psi^{(1, \ell)}(t,\cdot) -\widetilde{\psi}^{(1, \ell)}(t,\cdot) \|_{\dot{H}^{k}} \lesssim \ml{D}_{n,k+1+\ell}(t) \|(\psi_0 ,\psi_1 ,\psi_2 )\|_{\ml{A}_k}
		\end{align*}
		as $t\gg 1$ for $0 \leqslant k \leqslant s+2-\ell$ with $\ell=0,1,2$, where $\widetilde{C}=\widetilde{C}(\widetilde{A}_{0},\widetilde{A}_{1},\widetilde{\mathbb{B}})$ is a positive constant.
\end{description}
\end{theorem}
Once we have Proposition \ref{prop:4.2} and Corollary \ref{coro:4.2}, one may simply follow the proof of Theorem \ref{thm:2.3} with $A_{0}$, $A_{1}$ and $\mathbb{B}$ just replaced by $\widetilde{A}_{0}$, $\widetilde{A}_{1}$ and $\widetilde{\mathbb{B}}$, to obtain Theorem \ref{thm:4.2}.   
So, we omit the detail.

\section{Final remarks: Further asymptotic property for the MGT equation}\label{Section_Final_Remark}
Throughout the paper, we focused on the large-time convergences and profiles for the (Jordan-)MGT models in the viscous case. Due to the modeling that acoustic waves are originated from the Navier-Stokes equations coupled with either Fourier's law \eqref{Fourier-law} or Cattaneo's law \eqref{Cattaneo-law}, we will give some comments for the singular limits with respect to the thermal relaxation for the MGT equation \eqref{Eq_MGT}$_1$ and the linearized Kuznetsov's equation \eqref{Lienar_Kuznetsov}$_1$ under some datum assumptions.

Let us consider the inhomogeneous MGT equation in the Fourier space
\begin{align}\label{inhomo_Fourier}
	\begin{cases}
		\tau\hat{u}_{ttt}+\hat{u}_{tt}+|\xi|^2\hat{u}+(\delta+\tau)|\xi|^2\hat{u}_t=\hat{f}_{h},&\xi\in\mb{R}^n,\ t>0,\\
		\hat{u}(0,\xi)=\hat{u}_t(0,\xi)=\hat{u}_{tt}(0,\xi)=0,&\xi\in\mb{R}^n,
	\end{cases}
\end{align}
with $0<\tau\ll 1$. With the aid of energy method in the Fourier space associated with three equalities
\begin{align*}
	\frac{1}{2}\frac{\mathrm{d}}{\mathrm{d}t}\left(\left|\frac{1}{2}\hat{u}_t+\tau\hat{u}_{tt}\right|^2\right)&=-\frac{\tau}{2}|\hat{u}_{tt}|^2-\frac{\delta+\tau}{2}|\xi|^2|\hat{u}_t|^2-\frac{1}{4}\Re(\hat{u}_{tt}\bar{\hat{u}}_t)-\frac{1}{2}|\xi|^2\Re(\hat{u}\bar{\hat{u}}_t)\\
	&\quad-\tau(\delta+\tau)|\xi|^2\Re(\hat{u}_t\bar{\hat{u}}_{tt})-\tau|\xi|^2\Re(\hat{u}\bar{\hat{u}}_{tt})+\Re\left(\hat{f}_h\left(\frac{1}{2}\bar{\hat{u}}_t+\tau\bar{\hat{u}}_{tt}\right)\right),
\end{align*}
moreover,
\begin{align*}
	\frac{1}{2}\frac{\mathrm{d}}{\mathrm{d}t}\left(\frac{\tau}{\delta+\tau}|\xi|^2|\hat{u}+(\delta+\tau)\hat{u}_t|^2\right)&=\frac{\tau}{\delta+\tau}|\xi|^2\Re(\hat{u}_t\bar{\hat{u}})+\tau|\xi|^2|\hat{u}_t|^2+\tau|\xi|^2\Re(\hat{u}_{tt}\bar{\hat{u}})\\
	&\quad+\tau(\delta+\tau)|\xi|^2\Re(\hat{u}_{tt}\bar{\hat{u}}_t),
\end{align*}
and
\begin{align*}
	\frac{1}{2}\frac{\mathrm{d}}{\mathrm{d}t}\left(\frac{\delta-\tau}{2(\delta+\tau)}|\xi|^2|\hat{u}|^2+\frac{1}{4}|\hat{u}_t|^2\right)=\frac{\delta-\tau}{2(\delta+\tau)}|\xi|^2\Re(\hat{u}_t\bar{\hat{u}})+\frac{1}{4}\Re(\hat{u}_{tt}\bar{\hat{u}}_t),
\end{align*}
we are able to claim (only leaving the term with $|\hat{u}|^2$)
\begin{align}\label{Pointwise}
	\frac{\delta-\tau}{2(\delta+\tau)}|\xi|^2|\hat{u}(t,\xi)|^2\leqslant\frac{1+\tau(\delta-\tau)|\xi|^2}{8(\delta-\tau)|\xi|^2}\int_0^t|\hat{f}_h(s,\xi)|^2\mathrm{d}s.
\end{align}

Let us take $\hat{u}:=\widehat{\psi}-\widehat{\varphi}$ standing for the difference of solutions to the MGT equation \eqref{Eq_MGT} and Kuznetsov's equation \eqref{Lienar_Kuznetsov}$_1$ with
\begin{align}\label{Initial_Datum}
	\varphi(0,x)=\psi_0(x),\ \ \varphi_t(0,x)=\psi_1(x)\ \ \mbox{such that}\ \ \psi_2(x)=\Delta\psi_0(x)+\delta\Delta\psi_1(x).
\end{align}
So, we notice $f_{h}(t,x)=-\delta\tau\Delta\varphi_{tt}(t,x)$ in \eqref{inhomo_Fourier}. From the derived pointwise estimate \eqref{Pointwise}, one has
\begin{align*}
	|\widehat{\psi}(t,\xi)-\widehat{\varphi}(t,\xi)|^2=|\hat{u}(t,\xi)|^2\leqslant \tau^2 C(1+\tau|\xi|^2)\int_0^t|\widehat{\varphi}_{tt}(s,\xi)|^2\mathrm{d}s.
\end{align*}
Due to the fact that $\widehat{\varphi}_{tt}=-|\xi|^2(\widehat{\varphi}+\delta\widehat{\varphi}_t)$ of Kuznetsov's equation, and \cite[Lemma 2.4]{Ikehata-Natsume=2012} that
\begin{align*}
	|\xi|^2|\widehat{\varphi}(s,\xi)|+|\xi|\,|\widehat{\varphi}_t(s,\xi)|\leqslant C\mathrm{e}^{-c|\xi|^2\langle\xi\rangle^{-2}s}\left(|\xi|^2|\widehat{\psi}_0(\xi)|+|\xi|\,|\widehat{\psi}_1(\xi)|\right),
\end{align*}
we may obtain
\begin{align*}
	\|\widehat{\psi}(t,\cdot)-\widehat{\varphi}(t,\cdot)\|_{L^2}^2&\leqslant \tau^2 C\left(\int_{|\xi|\leqslant \varepsilon_0}|\xi|^4+\int_{|\xi|\geqslant \varepsilon_0}\langle\xi\rangle^6\right)\int_0^t\left(|\widehat{\varphi}_{t}(s,\xi)|^2+|\widehat{\varphi}(s,\xi)|^2\right) \mathrm{d}s \mathrm{d}\xi\\
	&\leqslant \tau^2 C\int_0^t(1+s)^{-1-\frac{n}{2}}\mathrm{d}s\left(\|\psi_0\|_{L^1}^2+\|\psi_1\|_{L^1}^2\right)+\tau^2 C\left(\|\psi_0\|_{H^4}^2+\|\psi_1\|_{H^3}^2\right)\\
	&\leqslant \tau C\left(\|\psi_0\|_{H^4\cap L^1}+\|\psi_1\|_{H^3\cap L^1}\right),
\end{align*}
where $C>0$ is independent of $\tau$. 

Finally, Parseval's formula shows
\begin{align}\label{Conv_1}	\sup\limits_{t\in[0,\infty)}\left\|\psi(t,\cdot)-\varphi(t,\cdot)\right\|_{L^2}\leqslant \tau C\left(\|\psi_0\|_{H^4\cap L^1}+\|\psi_1\|_{H^3\cap L^1}\right)
\end{align}
holding for $0<\tau\ll 1$ and any $n\geqslant 1$, which implies linear rate $\tau$ of the convergence in $L^{\infty}([0,\infty),L^2)$ between the MGT equation and the linearized Kuznetsov's equation if the condition \eqref{Initial_Datum} is valid. By the same way, we also can prove
\begin{align}\label{Conv_2}
	\sup\limits_{t\in[0,\infty)}\left\|\psi(t,\cdot)-\varphi(t,\cdot)\right\|_{L^{\infty}}\leqslant \tau C\left(\|\langle D\rangle^{s_0+2}\psi_0\|_{ L^1}+\|\langle D\rangle^{s_0+1}\psi_1\|_{ L^1}\right)
\end{align}
holding for $0<\tau\ll 1$  and any $n\geqslant 1$ with $s_0>n+2$. Together with \eqref{Conv_1} and \eqref{Conv_2}, it yields
\begin{align}\label{Conv_3}
	\psi(t,x)\to\varphi(t,x)\ \ \mbox{in}\ \ L^{\infty}([0,\infty),L^p)\ \ \mbox{for}\ \ 2\leqslant p\leqslant\infty
\end{align}
as $\tau\downarrow 0$, with the convergence rate $\tau$. This result completes the unknown case of strong convergence for $n=1,2$ in \cite[Theorem 4.1]{Chen-Ikehata=2021}. Furthermore, our result \eqref{Conv_3} holds for any $t>0$, namely, global (in time) convergence. However, the corresponding result for the JMGT equation \eqref{JMGT_Dissipative} in the framework of $H^s$ with suitable $s\geqslant0$ is still an open question since the disparity between weakly quasilinear evolution equation \eqref{Kuz-Eq} and semilinear hyperbolic equation \eqref{JMGT-ALL}.

\appendix
\section{Tools from Harmonic Analysis}\label{Appendix}
In this appendix, we will introduce interpolation theorem from Harmonic Analysis, which have been used in Section \ref{Section_GESDS_JMGT} to deal with the nonlinear terms.
\begin{prop}\label{fractionalgagliardonirenbergineq} (Fractional Gagliardo-Nirenberg inequality, \cite{Hajaiej-Molinet-Ozawa-Wang-2011})
	Let $p,p_0,p_1\in(1,\infty)$ and $\kappa\in[0,s)$ with $s\in(0,\infty)$. Then, for all $f\in L^{p_0}\cap \dot{H}^{s}_{p_1}$ the following inequality holds:
	\begin{align*}
		\|f\|_{\dot{H}^{\kappa}_{p}}\lesssim\|f\|_{L^{p_0}}^{1-\beta}\|f\|^{\beta}_{\dot{H}^{s}_{p_1}},
	\end{align*}
	where  $\beta=\beta_{\kappa,s}=\left(\frac{1}{p_0}-\frac{1}{p}+\frac{\kappa}{n}\right)\big/\left(\frac{1}{p_0}-\frac{1}{p_1}+\frac{s}{n}\right)$ and $ \beta\in\left[\frac{\kappa}{s},1\right]$.
\end{prop}

\begin{prop}\label{fractionleibnizrule} (Fractional Leibniz rule, \cite{Grafakos-Oh-2014})
	Let $s\in(0,\infty)$, $r\in[1,\infty]$, $p_1,p_2,q_1,q_2\in(1,\infty]$ satisfy the relation $\frac{1}{r}=\frac{1}{p_1}+\frac{1}{p_2}=\frac{1}{q_1}+\frac{1}{q_2}$. Then, for all $f\in\dot{H}^{s}_{p_1}\cap L^{q_1}$ and $g\in\dot{H}^{s}_{q_2}\cap L^{q_2}$
	the following inequality holds:
	\begin{align*}
			\|fg\|_{\dot{H}^{s}_{r}}\lesssim \|f\|_{\dot{H}^{s}_{p_1}}\|g\|_{L^{p_2}}+\|f\|_{L^{q_1}}\|g\|_{\dot{H}^{s}_{q_2}}.
	\end{align*}
\end{prop}

\begin{prop}
	Let $s >n/2$.
	Then the following estimate holds:
	\begin{align}\label{eq:C1}
		\| f \|_{L^{\infty}} \lesssim \| f \|_{L^{2}}^{1-\frac{n}{2s}} \|f\|_{\dot{H}^{s}}^{\frac{n}{2s}}.
	\end{align}
	In particular, when $n=1$, it holds that 
	\begin{align} \label{eq:C2}
		\| f \|_{L^{\infty}} \lesssim \| f \|_{L^{2}}^{\frac{1}{2}} \|f\|_{\dot{H}^{1}}^{\frac{1}{2}}.
	\end{align}
\end{prop}

\section*{Acknowledgments}
The authors thank Ya-guang Wang (Shanghai Jiao Tong University) and Ryo Ikehata (Hiroshima University) for the suggestions in the preparation of the paper.

\end{document}